\documentclass[onefignum,onetabnum]{siamonline171218}

\usepackage{amssymb}
\usepackage{subfig}
\usepackage{float,mathrsfs}
\usepackage[font=scriptsize]{caption}
\usepackage{lipsum}
\usepackage{amsfonts}
\usepackage{graphicx}
\usepackage{epstopdf}
\usepackage{algorithmic}
\ifpdf
  \DeclareGraphicsExtensions{.eps,.pdf,.png,.jpg}
\else
  \DeclareGraphicsExtensions{.eps}
\fi

\usepackage{enumitem}
\setlist[enumerate]{leftmargin=.5in}
\setlist[itemize]{leftmargin=.5in}

\newsiamremark{remark}{Remark}
\newsiamremark{hypothesis}{Hypothesis}
\crefname{hypothesis}{Hypothesis}{Hypotheses}
\newsiamthm{claim}{Claim}

\headers{A QMC Method for an Optimal Control Problem}{P.~A. Guth, V. Kaarnioja, F.~Y. Kuo, C. Schillings, I.~H. Sloan}

\title{A quasi-Monte Carlo Method for an Optimal Control Problem Under Uncertainty\thanks{\today.}}

\author{Philipp A. Guth\footnotemark[2]\thanks{Institute of Mathematics, University of Mannheim, 68159 Mannheim, Germany (\email{p.guth@mail.uni-mannheim.de},
\email{c.schillings@uni-mannheim.de}). PG is grateful to the DFG RTG1953 ``Statistical Modeling of Complex Systems and Processes'' for funding of this research. PG thanks the IPID4all/DAAD-Project ``Internationalization of Doctoral Education@the University of Mannheim'' for funding his research trip to University of New South Wales, Sydney.}
\and Vesa Kaarnioja\footnotemark[3]\thanks{School of Mathematics and Statistics, University of New South Wales, Sydney NSW 2052, Australia (\email{v.kaarnioja@unsw.edu.au},
\email{f.kuo@unsw.edu.au},
\email{i.sloan@unsw.edu.au}). VK, FK and IS gratefully acknowledge the financial support from the Australian Research Council (DP180101356).}
\and Frances Y. Kuo\footnotemark[3]
\and Claudia Schillings\footnotemark[2]
\and Ian H. Sloan\footnotemark[3]}

\usepackage{amsopn}

\ifpdf
\hypersetup{
  pdftitle={A quasi-Monte Carlo Method for an Optimal Control Problem Under Uncertainty},
  pdfauthor={P.~A. Guth, V. Kaarnioja, F.~Y. Kuo, C. Schillings, I.~H. Sloan}
}
\fi

\crefname{subsection}{Subsection}{}
\crefname{section}{Section}{}

\newcommand{\bsx}{{\pmb x}}
\newcommand{\bsy}{{\pmb y}}
\newcommand{\rd}{{\mathrm{d}}}
\newcommand{\bseta}{{\pmb \eta}}
\newcommand{\bsnu}{{\pmb \nu}}

\newcommand{\bsmu}{{\pmb \mu}}
\newcommand{\bsm}{{\pmb m}}

\begin{document}

\maketitle


\begin{abstract}
We study an optimal control problem under uncertainty, where the target function is the solution of an elliptic partial differential equation with random coefficients, steered by a control function. The robust formulation of the optimization problem is stated as a high-dimensional integration problem over the stochastic variables. It is well known that carrying out a high-dimensional numerical integration of this kind using a Monte Carlo method has a notoriously slow convergence rate; meanwhile, a faster rate of convergence can potentially be obtained by using sparse grid quadratures, but these lead to discretized systems that are non-convex due to the involvement of negative quadrature weights. In this paper, we analyze instead the application of a quasi-Monte Carlo method, which retains the desirable convexity structure of the system and has a faster convergence rate compared to ordinary Monte Carlo methods. In particular, we show that under moderate assumptions on the decay of the input random field, the error rate obtained by using a specially designed, randomly shifted rank-1 lattice quadrature rule is essentially inversely proportional to the number of quadrature nodes. The overall discretization error of the problem, consisting of the dimension truncation error, finite element discretization error and quasi-Monte Carlo quadrature error, is derived in detail. We assess the theoretical findings in numerical experiments.
\end{abstract}



\section{Introduction}
In this paper we consider an optimal control problem in the presence of uncertainty: the target function is the solution of an elliptic partial differential equation (PDE), steered by a control function, and having a random field as input coefficient. The random field is in principle infinite-dimensional, and in practice might need a large finite number of terms for accurate approximation. The novelty lies in the use and analysis of a specially designed quasi-Monte Carlo method to approximate the possibly high-dimensional integrals with respect to the stochastic variables.

Specifically, we consider the optimal control problem of finding 
\begin{equation}
\min_{z \in L^2(\Omega)} J(u,z)\,,\quad J(u,z) := \frac{1}{2} \int_{\Xi}\, \int_{\Omega} (u(\pmb x,\pmb y)-u_0(\pmb x))^2\, \mathrm d\pmb x\, \mathrm d\pmb y + \frac{\alpha}{2} \int_{\Omega} z(\pmb x)^2\, \mathrm d\pmb x\,, \label{eq:1.1}
\end{equation}
subject to the partial differential equation
\begin{align}
- \nabla \cdot (a(\pmb x,\pmb y) \nabla u(\pmb x,\pmb y)) &= z(\pmb x) \quad &\pmb x \in \Omega\,,&\quad \pmb y \in \Xi\,,\label{eq:1.2} \\
u(\pmb x,\pmb y) &= 0 \quad &\pmb x \in \partial \Omega\,,& \quad \pmb y\in \Xi\,, \label{eq:1.3} \\
z_{\min}(\pmb x) \leq&\ z(\pmb x) \leq z_{\max}(\pmb x) \quad &\text{a.e.~in}\ \Omega\,,& \label{eq:1.4}
\end{align}
for $\alpha > 0$ and a bounded domain $\Omega \subset \mathbb R^d$ with Lipschitz boundary $\partial \Omega$, where $d = 1,2$ or $3$. Further we assume
\begin{align}
u_0,z_{\min}, z_{\max} & \in L^2(\Omega)\,, \notag \\
z_{\min} \leq z_{\max} &\text{ a.e.~in } L^2(\Omega)\,. \label{eq:1.7}
\end{align}
Hence $\mathcal Z$, the set of \emph{feasible controls}, is defined by
\begin{align*}
\mathcal Z = \{ z \in L^2(\Omega)\ :\ z_{\min} \leq z \leq z_{\max}\quad \text{a.e.~in } \Omega \}\,.
\end{align*}
Note that $\mathcal Z$ is bounded, closed and convex and by \cref{eq:1.7} it is non-empty.

The gradients in \cref{eq:1.2} are understood to be with respect to the physical variable $\pmb x \in \Omega$, whereas $\pmb y \in \Xi$ is an infinite-dimensional vector $\pmb y = (y_j)_{j \geq 1}$ consisting of a countable number of parameters $y_j$, which are assumed to be independently and identically distributed (i.i.d.) uniformly in $[-\frac{1}{2},\frac{1}{2}]$ and we denote 
\begin{displaymath}
\Xi := \left[-\tfrac{1}{2},\tfrac{1}{2}\right]^{\mathbb N}\,.
\end{displaymath}
The parameter $\pmb y$ is then distributed on $\Xi$ with probability measure $\mu$, where
\begin{displaymath}
\mu(\mathrm d\pmb y) = \bigotimes_{j\geq1} \mathrm dy_j = \mathrm d\pmb y
\end{displaymath}
is the uniform probability measure on $\Xi$.

The input uncertainty is described by the parametric diffusion coefficient $a(\pmb x,\pmb y)$ in \cref{eq:1.2}, which is assumed to depend linearly on the parameters $y_j$, i.e.,
\begin{equation}
a(\pmb x,\pmb y) = \bar{a}(\pmb x) + \sum_{j\geq1} y_j\, \psi_j (\pmb x)\,, \quad \pmb x \in \Omega\,,\quad \pmb y \in \Xi\,. \label{eq:1.9}
\end{equation}
In order to ensure that the diffusion coefficient $a(\pmb x,\pmb y)$ is well defined for all $\pmb y \in \Xi$, we assume
\begin{align}
\bar{a} & \in L^{\infty}(\Omega)\,, \quad \sum_{j\geq 1}\ \|\psi_j\|_{L^{\infty}(\Omega)} < \infty\,. \label{eq:1.10}
\end{align}
Later in this article we shall impose a number of assumptions on the coefficients $a(\pmb x,\pmb y)$ as required.

A comprehensive overview of other possible formulations of the optimal control problem \cref{eq:1.1,eq:1.2,eq:1.3,eq:1.4} can be found, e.g., in \cite{AUH,BSSW}. They differ primarily in the computational cost and the robustness of the control with respect to the uncertainty. A lot of work \cite{AUH,ChenGhattas,KHRB,KunothSchwab2} has been done on formulations with stochastic controls, i.e., when the control depends directly on the uncertainty. Since practitioners often require a single deterministic control, the so-called robust deterministic formulation \cref{eq:1.1,eq:1.2,eq:1.3,eq:1.4} has received increasing attention in the recent past. This deterministic reformulation of the optimal control problem is based on a risk measure, such as the expected value, the conditional value-at-risk \cite{KouriSurowiec} or the combination of the expected value and the variance \cite{vanBarelVandewalle}.
Approaches to solve the resulting robust optimization problems include, e.g., Taylor approximation methods \cite{CVG}, sparse grids \cite{SparseGrids,KouriSurowiec} and multilevel Monte Carlo methods \cite{vanBarelVandewalle}. Multilevel Monte Carlo methods have first been analyzed for robust optimal control problems in the fundamental work \cite{vanBarelVandewalle}. Together with confirming numerical evidence, the theory in \cite{vanBarelVandewalle} shows the vast potential cost savings resulting from the application of multilevel Monte Carlo methods. Monte Carlo based methods do not require smoothness of the integrand with respect to the uncertain parameters. However, for many robust optimization problems, the integrands in the robust formulations are in fact smooth with respect to the uncertainty.

In this paper we propose the application of a quasi-Monte Carlo method to approximate the expected values with respect to the uncertainty. Quasi-Monte Carlo methods have been shown to perform remarkably well in the application to PDEs with random coefficients \cite{DKGNS, DKGS, GKNSSS, KKS, Kuo2016ApplicationOQ, KuoNuyens2018, KSSSU, Kuo2012QMCFEM, KSS2015, Schwab}. The reason behind their success is that it is possible to design quasi-Monte Carlo rules with error bounds not dependent on the number of uncertain variables, which achieve faster convergence rates compared to Monte Carlo methods in case of smooth integrands. In addition, quasi-Monte Carlo methods preserve the convexity structure of the optimal control problem due to their nonnegative (equal) quadrature weights. This work focuses on error estimates and convergences rates for the dimension truncation, the finite element discretization and the quasi-Monte Carlo quadrature, which are presented together with confirming numerical experiments.

This paper is structured as follows. The parametric weak formulation of the PDE problem is given in \cref{sec:2}. The corresponding optimization problem is discussed in \cref{sec:3}, with the unique solvability of the optimization problem considered in \cref{subsection31} and the requisite optimality conditions given in \cref{subsection32}. The gradient descent algorithm and its projected variant as they apply to our problem are presented in \cref{sec:Gradient descent} and \cref{subsection41}, respectively. The error analysis of \cref{sec:5} contains the main new theoretical results of this paper. \cref{subsection51} is concerned with the dimension truncation error, while \cref{section:FE discretization} addresses the finite element discretization error of the PDE problem. The regularity of the adjoint PDE problem is the topic of \cref{subsection54}, which leads to \cref{subsection:QMC} covering the quasi-Monte Carlo (QMC) integration error. \cref{section:optimalweights} details the design of optimally chosen weights for the QMC algorithm. Finally, the combined error and convergence rates for the PDE-constrained optimization problem are summarized in \cref{subsection56}.


\section{Parametric weak formulation}
\label{sec:2}

We state the variational formulation of the parametric elliptic boundary value problem \cref{eq:1.2,eq:1.3} for each value of the parameter $\pmb y \in \Xi$ together with sufficient conditions for the existence and uniqueness of solutions.

Our variational setting of \cref{eq:1.2} and \cref{eq:1.3} is based on the Sobolev space $H_0^1(\Omega)$ and its dual space $H^{-1}(\Omega)$ with the norm in $H_0^1(\Omega)$ defined by
\begin{align*}
\|v\|_{H_0^1(\Omega)} := \|\nabla v\|_{L^2(\Omega)}\,.
\end{align*}
The duality between $H_0^1(\Omega)$ and $H^{-1}(\Omega)$ is understood to be with respect to the pivot space $L^2(\Omega)$, which we identify with its own dual. We denote by $\langle \cdot,\cdot \rangle$ the $L^2(\Omega)$ inner product and the duality pairing between $H_0^1(\Omega)$ and $H^{-1}(\Omega)$. We introduce the continuous embedding operators $E_1: L^2(\Omega) \to H^{-1}(\Omega)$ and $E_2: H_0^1(\Omega) \to L^2(\Omega)$, with the embedding constants $c_1,c_2>0$ for the norms 
\begin{align}
\|v\|_{H^{-1}(\Omega)} &\leq c_1 \|v\|_{L^2(\Omega)} \label{c1}\,,\\
\|v\|_{L^2(\Omega)} &\leq c_2 \|v\|_{H_0^1(\Omega)} \label{c2}\,.
\end{align}

For fixed $\pmb y \in \Xi$, we obtain the following parameter-dependent weak formulation of the parametric deterministic boundary value problem  \cref{eq:1.2,eq:1.3}: 
for $\pmb y \in \Xi$ find $u(\cdot,\pmb y) \in H_0^1(\Omega)$ such that
\begin{equation}
\int_{\Omega} a(\pmb x,\pmb y) \nabla u(\pmb x,\pmb y)  \cdot \nabla v(\pmb x)\, \mathrm d\pmb x = \int_{\Omega} z(\pmb x) v(\pmb x)\, \mathrm d\pmb x\quad \forall v \in H_0^1(\Omega)\,. \label{eq:2.1}
\end{equation}
The parametric bilinear form $b(\pmb y;w,v)$ for $\pmb y \in \Xi$ is given by
\begin{equation}
b(\pmb y;w,v) := \int_{\Omega} a(\pmb x,\pmb y) \nabla w(\pmb x) \cdot \nabla v(\pmb x)\ \mathrm d\pmb x \quad \forall w,v \in H_0^1(\Omega)\,, \label{eq:2.2}
\end{equation}
allowing us to write the weak form of the PDE as
\begin{align}
b(\pmb y;u(\cdot,\pmb y),v) = \langle z,v \rangle \quad \forall v \in H_0^1(\Omega)\,.\label{eq:parametricweakproblem}
\end{align}

Throughout this paper we assume in addition to \cref{eq:1.9,eq:1.10} that
\begin{equation}
0 < a_{\min} \leq a(\pmb x,\pmb y) \leq a_{\max} <\infty\,, \quad \pmb x \in \Omega\,, \quad \pmb y \in \Xi\,,  \notag
\end{equation}
for some positive real numbers $a_{\min}$ and $a_{\max}$. Then the parametric bilinear form is continuous and coercive on $H_0^1(\Omega) \times H_0^1(\Omega)$, i.e., for all $\pmb y \in \Xi$ and all $w,v \in H_0^1(\Omega)$ we have
\begin{displaymath}
b(\pmb y;v,v) \geq a_{\min}\ \|v\|_{H_0^1(\Omega)}^2 \quad \text{ and } \quad |b(\pmb y;w,v)| \leq a_{\max}\ \|w\|_{H_0^1(\Omega)}\ \|v\|_{H_0^1(\Omega)}\,.
\end{displaymath}
With the Lax--Milgram lemma we may then infer that for every $z \in H^{-1}(\Omega)$ and given $\pmb y \in \Xi$, there exists a unique solution to the parametric weak problem: find $u(\cdot,\pmb y) \in H_0^1(\Omega)$ such that \cref{eq:parametricweakproblem} holds.
Hence we obtain the following result, which can also be found, e.g., in \cite{CohenDeVoreSchwab} and \cite{Kuo2012QMCFEM}.

\begin{theorem}\label{theorem:theorem1}
For every $z \in H^{-1}(\Omega)$ and every $\pmb y \in \Xi$, there exists a unique solution $u(\cdot,\pmb y) \in H_0^1(\Omega)$ of the parametric weak problem \cref{eq:2.1} (or equivalently, \cref{eq:parametricweakproblem}), which satisfies
\begin{equation}
\|u(\cdot,\pmb y)\|_{H_0^1(\Omega)} \leq \frac{\| z\|_{H^{-1}(\Omega)}}{a_{\min}}\,.\notag
\end{equation}
In particular, because of \cref{c1} it holds for $z \in L^2(\Omega)$ that
\begin{equation}
\|u(\cdot,\pmb y)\|_{H_0^1(\Omega)} \leq \frac{c_1\|z\|_{L^{2}(\Omega)}}{a_{\min}}\,. \label{eq:2.5}
\end{equation}
\end{theorem}


\section{The optimization problem}
\label{sec:3}

For the discussion of existence and uniqueness of solutions of the optimal control problem \cref{eq:1.1,eq:1.2,eq:1.3,eq:1.4}, we reformulate the problem to depend on $z$ only, a form often referred to as the \emph{reduced form of the problem}.

Due to \cref{c2} we can interpret the solution operator as a linear continuous operator with image in $L^2(\Omega)$, which leads to the following definition.

\begin{definition}\label{def:solutionoperator}
For arbitrary $\pmb y \in \Xi$ we call the unique mapping $S_{\pmb y}:L^2(\Omega) \to L^2(\Omega)$, which for every $\pmb y \in \Xi$ assigns to each $f \in L^2(\Omega)$ the unique solution $g \in L^2(\Omega)$ of the weak problem: find $g\in H_0^1(\Omega)$ such that
\begin{align*}
b(\pmb y;g,v) = \langle f,v \rangle \quad \forall v \in H_0^1(\Omega)\,.
\end{align*}
\end{definition}

Note that the solution operator $S_{\pmb y}$ depends on $\pmb y \in \Xi$ as indicated by the subscript. Further, $S_{\pmb y}$ is a self-adjoint operator, i.e., $S_{\pmb y}= S_{\pmb y}^*$, where $S_{\pmb y}^*$ is defined by $\langle S^*_{\pmb y} g,f \rangle = \langle g,S_{\pmb y}f \rangle$ $\forall f,g \in L^2(\Omega)$. The self-adjoint property holds since for all $f,g \in L^2(\Omega)$ we have $\langle S_{\pmb y}^* g, f\rangle = \langle g, S_{\pmb y}f\rangle = b(\pmb y; S_{\pmb y}g,S_{\pmb y}f) = \langle S_{\pmb y}g,f \rangle$. In the following we will omit the $*$ in $S^*_{\pmb y}$.

By \cref{def:solutionoperator} and \cref{eq:parametricweakproblem} it clearly holds that $u(\cdot,\pmb y) = S_{\pmb y}z$ for every $\pmb y \in \Xi$. Therefore we can write 
\begin{align*}
u(\cdot,\pmb y, z) := S_{\pmb y}z
\end{align*} 
as a function of $z$ and call it the \emph{state} corresponding to the control $z \in L^2(\Omega)$.  The optimal control problem then becomes a quadratic problem in the Hilbert space $L^2(\Omega)$: find
\begin{equation}
\min_{z\in \mathcal Z} J(z)\,, \quad J(z) := \frac{1}{2} \int_{\Xi} \|S_{\pmb y}z-u_0\|^2_{L^2(\Omega)}\ \mathrm d\pmb y + \frac{\alpha}{2} \|z\|^2_{L^2(\Omega)}\,. \label{eq:3.1}
\end{equation}


\subsection{Existence and uniqueness of solutions}
\label{subsection31}

Results on the existence of solutions for formulations of the optimization problem with stochastic controls, i.e., where it is assumed that the control $z$ is dependent on the parametric variable $\pmb y$, can be found, e.g., in \cite{ChenGhattas} and \cite{KunothSchwab2}. In \cite{KouriSurowiec} an existence result for solutions of a risk-averse PDE-constrained optimization problem is stated, where the objective is to minimize the conditional value-at-risk (CVaR).

\begin{theorem}
There exists a unique optimal solution $z^*$ of the problem \cref{eq:3.1}.
\end{theorem}
\begin{proof}
By assumption \cref{eq:1.7} there exists a $z_0 \in \mathcal Z$. For any $z \in \mathcal Z$ satisfying $\|z\|^2_{L^2(\Omega)} > \frac{2}{\alpha} J(z_0)$ it holds that
\begin{displaymath}
J(z) = \frac{1}{2} \int_{\Xi} \|S_{\pmb y}z - u_0\|^2_{L^2(\Omega)}\, \mathrm d\pmb y + \frac{\alpha}{2} \|z\|^2_{L^2(\Omega)} \geq \frac{\alpha}{2} \|z\|^2_{L^2(\Omega)} > J(z_0)\,.
\end{displaymath}
Hence, to find the optimal control $z^*$, we can restrict to the set $\widetilde{\mathcal Z} := \mathcal Z \cap \{ z \in L^2(\Omega) : \|z\|^2_{L^2(\Omega)} \leq \frac{2}{\alpha} J(z_0)\}$.
As $J(z) \geq 0$, the infimum $\widetilde J := \inf_{z \in \widetilde{\mathcal Z}} J(z)$ exists. Hence there exists a sequence $(z_i)_i \subset \widetilde{\mathcal Z}$ such that $J(z_i) \to \widetilde J$ as $i \to \infty$.
Since $\widetilde{\mathcal Z}$ is bounded, closed and convex it is weakly sequentially compact. Therefore there exists a subsequence $(z_{i_k})_k$, which converges weakly to $z^* \in \widetilde{\mathcal Z}$, i.e., $\langle z_{i_k},v \rangle \to \langle z^*,v \rangle$ $\forall v \in L^2(\Omega)$ as $k \to \infty$.
Since $\|S_{\pmb y}z-u_0\|^2_{L^2(\Omega)}$ as a function of $z$ is convex and continuous it is weakly lower semicontinuous. In consequence we have
\begin{displaymath}
\|S_{\pmb y}z^*-u_0\|^2_{L^2(\Omega)} \leq \liminf_{k \to \infty} \|S_{\pmb y}z_{i_k}-u_0\|^2_{L^2(\Omega)}\,.
\end{displaymath}
It follows that
\begin{align*}
J(z^*) &= \frac{1}{2} \int_{\Xi} \|S_{\pmb y}z^*-u_0\|^2_{L^2(\Omega)}\, \mathrm d\pmb y + \frac{\alpha}{2} \|z^*\|^2_{L^2(\Omega)}\\ 
&\leq \frac{1}{2}  \int_{\Xi} \liminf_{k \to \infty} \|S_{\pmb y}z_{i_k}-u_0\|^2_{L^2(\Omega)}\, \mathrm d\pmb y + \liminf_{k \to \infty} \frac{\alpha}{2} \|{z_{i_k}}\|^2_{L^2(\Omega)}\\
&\leq \liminf_{k \to \infty} J(z_{i_k}) = \widetilde J\,,
\end{align*}
where the last step follows by Fatou's lemma.
As $\widetilde J$ is the infimum of all possible values $J(z)$ and $z^* \in \widetilde{\mathcal Z}$, it follows that $J(z^*) = \widetilde J$ and hence $z^*$ is an optimal control. The uniqueness follows from the strict convexity of $J$. 
\end{proof}


\subsection{Optimality conditions}
\label{subsection32}

From standard optimization theory for convex $J$, we know that $z^*$ solves \cref{eq:3.1} if and only if the representer $J'$ of the Fr\'{e}chet derivative of $J$ satisfies the variational inequality $\langle J'(z^*),z-z^*\rangle \geq 0$ $\forall z \in \mathcal Z$.
It can be shown that
\begin{align}
J'(z) = \int_{\Xi} S_{\pmb y}(S_{\pmb y}z - u_0)\, \mathrm d\pmb y + \alpha z\,.\label{eq:gradient}
\end{align}
In the following we call $J'(z)$ the gradient of $J(z)$.

\begin{definition}\label{def:adjointstate}
For every $\pmb y \in \Xi$ and every $z \in \mathcal Z$, with $u(\cdot,\pmb y, z) = S_{\pmb y} z$ we call $q(\cdot,\pmb y,z) := S_{\pmb y}(S_{\pmb y}z-u_0) = S_{\pmb y}(u(\cdot,\pmb y ,z) - u_0)  \in L^2(\Omega)$ the adjoint state corresponding to the control $z$ and the state $u(\cdot, \pmb y,z)$.
\end{definition}
Note that $q(\cdot,\pmb y,z) \in L^2(\Omega)$ is by \cref{def:adjointstate}  the unique solution of the adjoint parametric weak problem: find $q(\cdot,\pmb y,z) \in H_0^1(\Omega)$ such that
\begin{align}
b(\pmb y;q(\cdot,\pmb y,z),w) = \langle (u(\cdot,\pmb y,z)-u_0),w \rangle \quad \forall w \in H_0^1(\Omega)\,,\label{eq:adjointparametricweakproblem}
\end{align}
where $u(\cdot,\pmb y,z)$ is the unique solution of 
\begin{align}
b(\pmb y; u(\cdot,\pmb y,z),v) = \langle z,v\rangle \quad \forall v \in H_0^1(\Omega)\,.\label{eq:33b}
\end{align}

The following result is a corollary to \cref{theorem:theorem1}.
\begin{corollary}\label{coro}
For every $z \in L^{2}(\Omega)$ and every $\pmb y \in \Xi$, there exists a unique solution $q(\cdot,\pmb y,z) \in H_0^1(\Omega)$ of the parametric weak problem \cref{eq:adjointparametricweakproblem}, which satisfies
\begin{equation}
\|q(\cdot,\pmb y,z)\|_{H_0^1(\Omega)} \leq \frac{c_1\|u(\cdot,\pmb y,z) - u_0\|_{L^2(\Omega)}}{a_{\min}} \leq C_q \left(\| z\|_{L^{2}(\Omega)} + \| u_0\|_{L^{2}(\Omega)}\right)\,, \label{coro:2.4}
\end{equation}
where $C_q := \max\left(\frac{c_1}{a_{\min}}, \frac{c_1^2c_2}{a^2_{\min}}\right)$ and $c_1,c_2> 0$ are the embedding constants in \cref{c1,c2}.
\end{corollary}

As a consequence of \cref{eq:gradient} and \cref{def:adjointstate} we get \begin{align}
J'(z) = \int_{\Xi} q(\cdot,\pmb y,z)\, \mathrm d\pmb y + \alpha z\,,\label{eq:gradient3}
\end{align}
which directly leads to the following result. 
\begin{lemma} \label{theorem:3.4}
A control $z^* \in \mathcal Z$ solves \cref{eq:3.1} if and only if
\begin{align} 
\left\langle \int_{\Xi} q(\cdot,\pmb y,z^*)\, \mathrm d\pmb y + \alpha z^*, z - z^*\right\rangle \geq 0 \quad \forall z \in \mathcal Z\,, \label{eq:gradient2}
\end{align}
where $q(\cdot,\pmb y,z^*)$ is the adjoint state corresponding to $z^*$.
\end{lemma}

The variational inequality $\langle J'(z^*),z-z^*\rangle \geq 0$ $\forall z \in \mathcal Z$ holds if and only if there exist a.e.~nonnegative functions $\mu_a,\mu_b \in L^2(\Omega)$ such that $J'(z^*) - \mu_a + \mu_b = 0$ and that the complementary constraints $(z^* - z_{\min})\mu_a  = (z_{\max} - z^*)\mu_b = 0$ are satisfied a.e.~in $\Omega$, cf. \cite[Theorem 2.29]{Troeltzsch}. Thus we obtain the following KKT-system.

\begin{theorem}\label{theorem:variational inequality}
A control $z^* \in L^2(\Omega)$ is the unique minimizer of \cref{eq:3.1} if and only if it satisfies the following KKT-system:
\begin{align}\label{eq:KKT}
\begin{cases}
- \nabla \cdot (a(\pmb x,\pmb y) \nabla u(\pmb x,\pmb y,z^*)) = z^*(\pmb x)  &\pmb x \in\Omega\,, \quad \pmb y \in \Xi\,,\\
u(\pmb x,\pmb y,z^*) = 0 &\pmb x \in \partial \Omega\,, \quad \!\!\! \pmb y \in \Xi\,,\\[.7em]
- \nabla \cdot (a(\pmb x,\pmb y) \nabla q(\pmb x,\pmb y,z^*)) = u(\pmb x,\pmb y,z^*) - u_0(\pmb x) &\pmb x \in\Omega\,, \quad \pmb y \in \Xi\,,\\
q(\pmb x,\pmb y,z^*) = 0 &\pmb x \in \partial \Omega\,, \quad \!\!\! \pmb y \in \Xi\,,\\[.7em]
\displaystyle \int_{\Xi} q(\pmb x,\pmb y,z^*)\, \mathrm d\pmb y + \alpha z^*(\pmb x) - \mu_a(\pmb x) + \mu_b(\pmb x) = 0 \quad &\pmb x \in\Omega\,,\\[.7em]
z_{\min}(\pmb x) \leq z^*(\pmb x) \leq z_{\max}(\pmb x)\,, \quad \mu_a(\pmb x) \geq 0\,, \quad \mu_b(\pmb x) \geq 0\,, &\pmb x \in\Omega\,,\\
(z^*(\pmb x) - z_{\min}(\pmb x))\mu_a(\pmb x) = (z_{\max}(\pmb x) - z^*(\pmb x))\mu_b(\pmb x) = 0\,, &\pmb x \in\Omega\,.
\end{cases}
\end{align}
\end{theorem}


\section{Gradient descent algorithms}
We present a gradient descent algorithm to solve the optimal control problem for the case without control constraints ($\mathcal Z = L^2(\Omega)$) in \cref{sec:Gradient descent} and a projected variant of the algorithm for the problem with control constraints in \cref{subsection41}.

\subsection{Gradient descent}
\label{sec:Gradient descent}

Consider problem \cref{eq:3.1} with $z_{\min} = -\infty$ and $z_{\max} = \infty$, i.e., $\mathcal Z = L^2(\Omega)$. Then $\mu_{a} = 0 = \mu_{b}$ and $z^*$ is unique minimizer of \cref{eq:3.1} if and only if $J'(z^*) = 0$. 
To find the minimizer $z^*$ of $J$ we use the gradient descent method, for which the descent direction is given by the negative gradient $-J'$, see \cref{alg:gradient descent}.

\begin{algorithm}[t]
\caption{Gradient descent}
\label{alg:gradient descent}
Input: starting value $z \in L^2(\Omega)$
\begin{algorithmic}[1]
\WHILE{$\| J'(z)\|_{L^2(\Omega)} >$TOL}
\STATE\label{444}{find step size $\eta$ using \cref{alg:Armijo}}
\STATE{set $z := z - \eta J'(z)$}
\ENDWHILE
\end{algorithmic}
\end{algorithm}
\begin{algorithm}[t]
\caption{Armijo rule}
\label{alg:Armijo}
Input: current $z$, parameters $\beta,\gamma \in (0,1)$\\
Output: step size $\eta > 0$
\begin{algorithmic}[1]
\STATE{set $\eta := 1$}
\WHILE{$J(z - \eta J'(z))- J(z) > - \eta \gamma \|J'(z)\|^2_{L^2(\Omega)}$}
\STATE{set $\eta := \beta \eta$}
\ENDWHILE
\end{algorithmic}
\end{algorithm}

Note that in every iteration in \cref{alg:gradient descent} several evaluations of $q$ are required in order to approximate the infinite-dimensional integral $\int_{\Xi}\ q(\cdot,\pmb y,z)\ \mathrm d\pmb y$ in the gradient of $J$, see \cref{eq:gradient2}. Further, for each evaluation of $q$ one needs to solve the state PDE and the adjoint PDE.

\begin{theorem}\label{theorem:convergenceGradDesc2}
For arbitrary starting values $z_0 \in L^2(\Omega)$ and $z_{\min} = -\infty$ and $z_{\max} = \infty$, the sequence $\{z_i\}$ generated by \cref{alg:gradient descent} satisfies $J'(z_i) \to 0$ as $i \to \infty$ and the sequence converges to the unique solution $z^*$ of \cref{eq:3.1}.
\end{theorem}

\begin{proof}
The first part is shown in \cite[Theorem 2.2]{HinzePinnauUlbrich}. Now let $z^*$ be the unique solution of \cref{eq:3.1}. Then
\begin{align*}
\alpha \| z_i - z^* \|_{L^2(\Omega)}^2 &\leq \int_{\Xi} \|S_{\pmb y}(z_i-z^*)\|_{L^2(\Omega)}^2\, \mathrm d\pmb y + \alpha \|z_i - z^*\|_{L^2(\Omega)}^2\\ 
&= \left\langle z_i-z^*, \int_{\Xi} (S_{\pmb y}S_{\pmb y} + \alpha I) (z_i-z^*)\, \mathrm d\pmb y \right\rangle\\
&= \left\langle z_i-z^*, \int_{\Xi} \left((S_{\pmb y}S_{\pmb y} + \alpha I)z_i - S_{\pmb y}u_0 \right) \mathrm d\pmb y \right\rangle\\
&= \left\langle z_i-z^*, J'(z_i) \right\rangle
\leq \|z_i-z^*\|_{L^2(\Omega)} \|J'(z_i)\|_{L^2(\Omega)}\,,
\end{align*}
where we used Fubini's Theorem in the first equality and then $\int_\Xi q(\cdot,\pmb y,z^*)\, \mathrm d\pmb y = -\alpha z^*$.
Hence we obtain
\begin{align*}
\|z_i-z^*\|_{L^2(\Omega)} \leq \frac{1}{\alpha} \|J'(z_i)\|_{L^2(\Omega)} \to 0\quad \text{as } i \to \infty\,,
\end{align*}
and thus $z_i \to z^*$ in $L^2(\Omega)$. Further, by continuity of $J$ it follows that $J(z_i) \to J(z^*)$.
\end{proof}


\subsection{Projected gradient descent}
\label{subsection41}

Consider now problem \cref{eq:3.1} with
\begin{align*}
-\infty < z_{\min} < z_{\max} < \infty\quad \text{a.e.~in}\ \Omega \,,
\end{align*} 
i.e., $\mathcal Z \subsetneq L^2(\Omega)$.
The application of \cref{alg:gradient descent} to feasible $z_i$ might lead to infeasibility of $z_i - \eta J'(z_i)$ even for small stepsizes $\eta >0$. On the other hand, considering only those $\eta >0$ for which $z_i - \eta J'(z_i)$ stays feasible is not viable since this might result in very small step sizes $\eta$.

To incorporate these constraints we use the projection $P_{\mathcal Z}$ onto $\mathcal Z$ given by
\begin{equation}\label{eq:projection}
P_{\mathcal Z}(z)(\pmb x) = P_{[z_{\min}(\pmb x),z_{\max}(\pmb x)]}(z(\pmb x)) = \max(z_{\min}(\pmb x),\min(z(\pmb x),z_{\max}(\pmb x)))\,,
\end{equation}
and perform a line search along the projected path $\{P_{\mathcal{Z}}(z_i - \eta J'(z_i)):\ \eta > 0\}$.
One can show (\cite[Lemma 1.10]{HinzePinnauUlbrich}) that the variational inequality \cref{eq:gradient2} is equivalent to $z^* - P_{\mathcal Z}(z^* - J'(z^*))  = 0$.
This leads to \cref{alg:projected gradient descent}, which is justified by \cref{theorem:projgraddesc}.
\begin{algorithm}[t]
\caption{Projected gradient descent}
\label{alg:projected gradient descent}
Input: feasible starting value $z \in \mathcal Z$
\begin{algorithmic}[1]
\WHILE{$\| z - P_{\mathcal Z}(z - J'(z)) \|_{L^2(\Omega)} >$TOL}
\STATE{find step size $\eta$ using \cref{alg:projected Armijo}}
\STATE{set $z := P_{\mathcal Z}(z - \eta J'(z))$}
\ENDWHILE
\end{algorithmic}
\end{algorithm}
\begin{algorithm}[t]
\caption{Projected Armijo rule}
\label{alg:projected Armijo}
Input: current $z$, parameters $\beta,\gamma \in (0,1)$\\
Output: step size $\eta > 0$
\begin{algorithmic}[1]
\STATE{set $\eta := 1$}
\WHILE{$J(P_{\mathcal Z}(z - \eta J'(z)))- J(z) > -\frac{\gamma}{\eta} \|z-P_{\mathcal Z}(z - \eta J'(z))\|^2_{L^2(\Omega)}$}
\STATE{set $\eta := \beta \eta$}
\ENDWHILE
\end{algorithmic}
\end{algorithm}

\begin{theorem}\label{theorem:projgraddesc}
For feasible starting values $z_0 \in \mathcal Z$, the sequence $\{z_i\}$ generated by \cref{alg:projected gradient descent} satisfies
\begin{align*}
\lim_{i \to \infty} \|z_i - P_{\mathcal Z}(z_i - J'(z_i))\|_{L^2(\Omega)} = 0\,,
\end{align*}
where $P_{\mathcal Z}$ is defined by \cref{eq:projection}. Moreover, the sequence $\{z_i\}$ converges to the unique solution $z^*$ of \cref{eq:3.1}.
\end{theorem}

\begin{proof}
For the proof of the first result we refer to \cite[Theorem 2.4]{HinzePinnauUlbrich}. 
By construction $J(z_i)$ is monotonically decreasing in $i$ and $J(z) \geq 0$ for all $z \in L^2(\Omega)$. Thus we know $\lim_{i \to \infty} J(z_i) =  \widetilde J = \inf_{i \in \mathbb N} J(z_i)$. Together with the projected Armijo rule in \cref{alg:projected Armijo} this further implies
\begin{align*}
J(z_{i+1}) - J(z_{i}) \leq -\frac{\gamma}{\eta_i}\|z_i -P_{\mathcal Z}(z_i - \eta_i J'(z_i))\|_{L^2(\Omega)}^2 =  -\frac{\gamma}{\eta_i}\|z_{i} - z_{i+1}\|_{L^2(\Omega)}^2 \to 0\,,
\end{align*}
and thus $z_i \to \widetilde z$ in $L^2(\Omega)$ as $i \to \infty$, for some $\widetilde z \in \mathcal Z$. 
By continuity of $\|\cdot\|_{L^2(\Omega)}$, continuity of $J'(\cdot)$ and continuity of $P_{\mathcal Z}(\cdot)$ we know
\begin{align*}
\lim_{i \to \infty} \|z_i - P_{\mathcal Z}(z_i - J'(z_i)) \|_{L^2(\Omega)} &= \|\lim_{i \to \infty} z_i - P_{\mathcal Z}(\lim_{i \to \infty} z_i -  J'(\lim_{i \to \infty}z_i) )\|_{L^2(\Omega)}\\
&= \|\widetilde z - P_{\mathcal Z}(\widetilde z +  J'(\widetilde z)) \|_{L^2(\Omega)} = 0\,,
\end{align*}
which is equivalent to 
\begin{align*}
\widetilde z - P_{\mathcal Z}(\widetilde z - J'(\widetilde z)) = 0\,.
\end{align*}
Thus $\widetilde z$ satisfies the variational inequality \cref{eq:gradient2} and is the unique minimizer $z^*$ of \cref{eq:3.1}.
\end{proof}


\section{Discretization of the problem and error expansion}
\label{sec:5}

In the following we consider an approximation/discretization of problem \cref{eq:1.1,eq:1.2,eq:1.3,eq:1.4}. 
Given $s \in \mathbb N$ and $\pmb y \in \Xi$, we notice that truncating the sum in \cref{eq:1.9} after $s$ terms is the same as setting $y_j = 0$ for $j \geq s+1$.
For every $\pmb y \in \Xi$ we denote the unique solution of the parametric weak problem \cref{eq:33b} corresponding to the dimensionally truncated diffusion coefficient $a(\cdot,(y_1,y_2,\ldots,y_s,0,0,\ldots))$ by $u_s(\cdot ,\pmb y,z) := u(\cdot , (y_1,y_2,\ldots,y_s,0,0,\ldots),z)$. Similarly we write $q_s(\cdot,\pmb y,z) := q(\cdot , (y_1,y_2,\ldots,y_s,0,0,\ldots),z)$ for any $\pmb y \in \Xi$ for the unique solution of the adjoint parametric weak problem \cref{eq:adjointparametricweakproblem} corresponding to the dimensionally truncated diffusion coefficient and truncated right-hand side $u_s(\cdot,\pmb y,z)-u_0$.

We further assume that we have access only to a finite element discretization $u_{s,h}(\cdot,\pmb y,z)$ of the truncated solution to \cref{eq:33b}, to be defined precisely in \cref{section:FE discretization}, and we write $q_{s,h}(\cdot,\pmb y,z)$ for the truncated adjoint state corresponding to $u_{s,h}(\cdot,\pmb y,z)$.

By abuse of notation we also write $u_s(\cdot,\pmb y,z) = u_s(\cdot,\pmb y_{\{1:s\}},z) =S_{\pmb y_{\{1:s\}}}z$ and $q_s(\cdot,\pmb y,z) =q_s(\cdot,\pmb y_{\{1:s\}},z)$ in conjunction with $u_{s,h}(\cdot,\pmb y,z) = u_{s,h}(\cdot,\pmb y_{\{1:s\}},z) = S_{\pmb y_{\{1:s\}},h}z$ and $q_{s,h}(\cdot,\pmb y,z)= q_{s,h}(\cdot,\pmb y_{\{1:s\}},z)$ for $s$-dimensional $\pmb y_{\{1:s\}} \in \Xi_s:= \left[-\frac{1}{2},\frac{1}{2}\right]^s$. Here and in the following $\{1:s\}$ is a shorthand notation for the set $\{1,2,\ldots,s\}$ and $\pmb y_{\{1:s\}}$ denotes the variables $y_j$ with $j \in \{1:s\}$.

Finally we use an $n$-point quasi-Monte Carlo approximation for the integral over $\Xi_s$ leading to the following discretization of \cref{eq:3.1}
\begin{align}
\min_{z \in \mathcal Z}  J_{s,h,n}(z)\,, \quad J_{s,h,n}(z) := \frac{1}{2n} \sum_{i=1}^{n} \| S_{\pmb y^{(i)},h} z -u_0\|^2_{L^2(\Omega)} + \frac{\alpha}{2} \| z\|^2_{L^2(\Omega)}\,, \label{eq:QMCFEobjective}
\end{align}
for quadrature points $\pmb y^{(i)} \in \Xi_s$, $i \in \{1,\ldots,n\}$, to be defined precisely in \cref{subsection:QMC}.

In analogy to \cref{eq:gradient3} it follows that the gradient of $J_{s,h,n}$, i.e., the representer of the Fr\'{e}chet derivative of $J_{s,h,n}$ is given by
\begin{align*}
J'_{s,h,n}(z) = \frac{1}{n} \sum_{i=1}^n q_{s,h}(\cdot,\pmb y^{(i)},z) + \alpha z\,.
\end{align*}
Due to the positive weights of the quadrature rule, \cref{eq:QMCFEobjective} is still a convex minimization problem. Existence and uniqueness of the solution $z^{*}_{s,h,n}$ of \cref{eq:QMCFEobjective} follow by the previous arguments. Quasi-Monte Carlo methods are designed to have convergence rates superior to Monte Carlo methods. Other candidates for obtaining faster rates of convergence include, e.g., sparse grid methods, but the latter involve negative weights, meaning that the corresponding discretized optimization problem will be generally non-convex, see, e.g., \cite{SparseGrids}.

\begin{theorem}\label{theorem:theorem51}
Let  $z^*$ be the unique minimizer of \cref{eq:3.1} and let $z^{*}_{s,h,n}$ be the unique minimizer of \cref{eq:QMCFEobjective}. It holds that
\begin{align}\label{eq:convergence}
\|z^*-z^{*}_{s,h,n}\|_{L^2(\Omega)} \leq \frac{1}{\alpha} \left\| \int_{\Xi} q(\cdot,\pmb y,z^*)\, \mathrm d\pmb y - \frac{1}{n} \sum_{i = 1}^{n} q_{s,h}(\cdot,\pmb y^{(i)},z^*)\ \right\|_{L^2(\Omega)} \,,
\end{align}
for quadrature points $\pmb y^{(i)} \in \left[-\frac{1}{2},\frac{1}{2}\right]^s$, $i \in \{1,\ldots,n\}$.
\end{theorem}

\begin{proof}
By the optimality of $z_{s,h,n}^*$ it holds for all $z \in \mathcal Z$ that 
$\langle  J'_{s,h,n}(z_{s,h,n}^{*}), z - z_{s,h,n}^{*} \rangle \geq 0$
and thus in particular $\langle  J'_{s,h,n}(z_{s,h,n}^{*}), z^* - z_{s,h,n}^{*} \rangle \geq 0$.
Similarly it holds for all $z \in \mathcal Z$ that 
$\langle  J'(z^{*}), z - z^{*} \rangle \geq 0$ and thus in particular $\langle -J'(z^*) , z^* - z_{s,h,n}^{*} \rangle \geq 0$.
Adding these inequalities leads to
\begin{align*}
\langle  J'_{s,h,n} (z_{s,h,n}^{*}) - J'(z^*), z^* -z_{s,h,n}^{*} \rangle \geq 0\,.
\end{align*}
Thus
\begin{align*}
\alpha \|z^* - z^{*}_{s,h,n} \|^2_{L^2(\Omega)} &\leq \alpha \| z^* - z^{*}_{s,h,n} \|^2_{L^2(\Omega)} + \left\langle  J'_{s,h,n}(z^{*}_{s,h,n}) - J'(z^*) , z^* - z^{*}_{s,h,n} \right\rangle\\
&=\left\langle J'_{s,h,n}(z^{*}_{s,h,n}) - \alpha z^{*}_{s,h,n} - J'(z^*) + \alpha z^*, z^* - z^*_{s,h,n} \right\rangle\\
&=\left\langle J'_{s,h,n}(z^{*}_{s,h,n}) - \alpha z^{*}_{s,h,n} - J'_{s,h,n}(z^*) + \alpha z^*, z^* - z^*_{s,h,n} \right\rangle\\
&\quad + \left\langle J'_{s,h,n}(z^*) - \alpha z^* - J'(z^*) + \alpha z^*, z^* - z^*_{s,h,n} \right\rangle\\
&=- \frac{1}{n} \sum_{i = 1}^{n} \|u_{s,h}(\cdot,\pmb y^{(i)},z_{s,h,n}^*) -u_{s,h}(\cdot,\pmb y^{(i)},z^*)\|^2_{L^2(\Omega)}\\
&\quad + \left\langle \frac{1}{n}\sum_{i=1}^nq_{s,h}(\cdot, \pmb y^{(i)},z^*) - \int_{\Xi} q(\cdot,\pmb y, z^*)\, \mathrm d\pmb y, z^* - z^*_{s,h,n} \right\rangle\\
&\leq  \left\|\frac{1}{n}\sum_{i=1}^n q_{s,h}(\cdot, \pmb y^{(i)},z^*) - \int_{\Xi} q(\cdot,\pmb y, z^*)\, \mathrm d\pmb y\right\|_{L^2(\Omega)} \| z^* - z^*_{s,h,n}\|_{L^2(\Omega)}\,,
\end{align*}
where in the fourth step we used the fact that
$J'_{s,h,n}(z^{*}_{s,h,n}) - \alpha z^{*}_{s,h,n} - J'_{s,h,n}(z^*) + \alpha z^*=\frac{1}{n} \sum_{i=1}^n (S_{\pmb y^{(i)},h}(S_{\pmb y^{(i)},h}z_{s,h,n}^*) - S_{\pmb y^{(i)},h}(S_{\pmb y^{(i)},h}z^*))$ together with the self-adjointness of the operator $S_{\pmb y^{(i)},h}$ in order to obtain
$
\frac{1}{n}\sum_{i=1}^n\langle S_{\pmb y^{(i)},h}(S_{\pmb y^{(i)},h}z_{s,h,n}^*) -S_{\pmb y^{(i)},h}(S_{\pmb y^{(i)},h}z^*), z^* - z^*_{s,h,n} \rangle
=-\frac{1}{n}\sum_{i=1}^n\| u_{s,h}(\cdot,\pmb y^{(i)},z_{s,h,n}^*) -u_{s,h}(\cdot,\pmb y^{(i)},z^*)\|_{L^2(\Omega)}\,.
$
The result then follows from $\alpha > 0$. 
\end{proof}
 
We can split up the error on the right-hand side in \cref{eq:convergence} into dimension truncation error, FE discretization error and QMC quadrature error as follows 
\begin{align}\label{eq:errorexpansion}
\int_{\Xi} q(\pmb x,\pmb y,z)\, \mathrm d\pmb y - \frac{1}{n} \sum_{i = 1}^{n} q_{s,h}(\pmb x,\pmb y^{(i)},z)
&= \underbrace{\int_{\Xi} \left(q(\pmb x,\pmb y,z) - q_s(\pmb x,\pmb y,z)\right)\, \mathrm d\pmb y}_{\text{truncation error}}\\ 
&+  \underbrace{\int_{\Xi_s} \left( q_s(\pmb x,\pmb y_{\{1:s\}},z) - q_{s,h}(\pmb x,\pmb y_{\{1:s\}},z)\right)\, \mathrm d \pmb y_{\{1:s\}}}_{\text{FE discretization error}}\notag\\
& + \underbrace{\int_{\Xi_s} q_{s,h}(\pmb x,\pmb y_{\{1:s\}},z)\,  \mathrm d \pmb y_{\{1:s\}} - \frac{1}{n} \sum_{i = 1}^{n} q_{s,h}(\pmb x,\pmb y^{(i)},z)}_{\text{QMC quadrature error}}.\notag
\end{align}
These errors can be controlled as shown in \cref{theorem:Truncationerror}, \cref{theorem:FEapprox} and \cref{theorem:QMCerror} below. The errors will be analysed separately in the following subsections.


\subsection{Truncation error}
\label{subsection51}

The proof of the following theorem is motivated by \cite{Gantner}. However, in this paper we do not apply a bounded linear functional to the solution of the PDE $q(\cdot,\pmb y,z)$. Moreover, the right-hand side $u(\cdot,\pmb y,z) - u_0$ of the adjoint PDE depends on the parametric variable $\pmb y$. Further, we do not need the explicit assumption that the fluctuation operators $B_j$ (see below) are small with respect to the mean field operator $A(\pmb 0)$ (see below), i.e., $\sum_{j \geq 1}\|A^{-1}(\pmb 0)B_j\|_{\mathcal L(H_0^1(\Omega))} \leq \kappa < 2$, cf. \cite[Assumption 1]{Gantner}. Here and in the following $\mathcal L(H_0^1(\Omega))$ denotes the space of all bounded linear operators in $H_0^1(\Omega)$. For these reasons the proof of our result differs significantly from the proof in \cite{Gantner}.

To state the proof of the subsequent theorem, we introduce the following notation:
for a multi-index $\pmb \nu = (\nu_j)_j$ with $\nu_j \in \{0,1,2,\ldots\}$, we denote its order $|\pmb \nu| := \sum_{j \geq1} \nu_j$ and its support as $\text{supp}(\pmb \nu) := \{j \geq 1: \nu_j \geq 1\}$. Furthermore, we denote the countable set of all finitely supported multi-indices by
\begin{displaymath}
\mathcal D := \{ \pmb \nu \in \mathbb N_0^{\infty} : \left|\text{supp}(\pmb \nu)\right| < \infty\}\,.
\end{displaymath} 
Let $b_j$ be defined by
\begin{align}
b_j := \frac{\|\psi_j\|_{L^{\infty}(\Omega)}}{a_{\min}}, \ j \geq 1\,.\label{b_j}
\end{align}
Then we write $\pmb b := (b_j)_{j\geq 1}$ and $\pmb b^{\pmb \nu} := \prod_{j \geq 1} b_j^{\nu_j}$.

\begin{theorem}[Truncation error]\label{theorem:Truncationerror}
Assume there exists $0<p<1$ such that
\begin{align*}
\sum_{j\geq 1}\|\psi_j\|_{L^{\infty}(\Omega)}^p < \infty\,.
\end{align*}
In addition let the $\psi_j$ be ordered such that $\|\psi_j\|_{L^{\infty}(\Omega)}$ are nonincreasing:
\begin{align*}
\|\psi_1\|_{L^{\infty}(\Omega)} \geq \|\psi_2\|_{L^{\infty}(\Omega)} \geq \|\psi_3\|_{L^{\infty}(\Omega)} \geq \cdots\,.
\end{align*}
Then for $z \in L^2(\Omega)$, for every $\pmb y \in \Xi$, and every $s \in \mathbb N$, the truncated adjoint solution $q_s(\cdot,\pmb y,z)$ satisfies
\begin{align}\label{eq:truncationerror1}
\left\| \int_{\Xi} \left(q(\cdot,\pmb y,z) - q_s(\cdot,\pmb y,z)\right)\, \mathrm d\pmb y \right\|_{L^2(\Omega)} \leq C \left(\|z\|_{L^2(\Omega)}+ \|u_0\|_{L^2(\Omega)}\right) s^{-\left(\frac{2}{p}-1\right)}\,,
\end{align}
for some constant $C > 0$ independent of $s$, $z$ and $u_0$.
\end{theorem}

\begin{proof}
First we note that the result holds trivially without the factor $s^{-(2/p-1)}$ in the error estimate. This is true since \cref{coro} holds for the special case $\pmb y = (y_1,y_2,\ldots,y_s,0,0,\ldots)$, and so \cref{eq:truncationerror1} holds without the factor $s^{-(2/p-1)}$ for $C = 2c_2C_q$.
As a consequence, it is sufficient to prove the result for sufficiently large $s$, since it will then hold for all $s$, by making, if necessary, an obvious adjustment of the constant.
To this end we define $A = A(\pmb y): H_0^1(\Omega) \to H^{-1}(\Omega)$ by
\begin{align*}
\langle A(\pmb y)w,v\rangle := b(\pmb y; w,v) \quad \forall v,w \in H_0^1(\Omega)\,,
\end{align*}
and $A_s(\pmb y) := A((y_1,y_2,\ldots,y_s,0,0,\ldots))$. Both $A(\pmb y)$ and $A_s(\pmb y)$ are boundedly invertible operators from $H_0^1(\Omega)$ to $H^{-1}(\Omega)$ since
for all $\pmb y \in \Xi$ it holds for $w \in H_0^1(\Omega)$ and $z \in H^{-1}(\Omega)$ that
\begin{align*}
\|A(\pmb y) w\|_{H^{-1}(\Omega)} = \sup_{v \in H_0^1(\Omega)} \frac{\langle A(\pmb y) w, v \rangle}{\|v\|_{H_0^1(\Omega)}} = \sup_{v \in H_0^1(\Omega)} \frac{b(\pmb y;w,v)}{\|v\|_{H_0^1(\Omega)}} \leq a_{\max} \|w\|_{H_0^1(\Omega)}\,,
\end{align*}
together with a similar bound for $A_s(\pmb y)$; and in the reverse direction, from \cref{theorem:theorem1}
\begin{align}
\|A^{-1}(\pmb y) z\|_{H_0^1(\Omega)} \leq \frac{\|z\|_{H^{-1}(\Omega)}}{a_{\min}}\,, \qquad \|A_s^{-1}(\pmb y) z\|_{H_0^1(\Omega)} \leq \frac{\|z\|_{H^{-1}(\Omega)}}{a_{\min}}\,. \label{eq:Amin}
\end{align}
It follows that the solution operator $S_{\pmb y}: L^2(\Omega) \to L^2(\Omega)$ defined in \cref{def:solutionoperator} can be written as $S_{\pmb y} = E_2 A^{-1}(\pmb y)E_1$, where $E_1: L^2(\Omega) \to H^{-1}(\Omega)$ and $E_2: H_0^1(\Omega) \to L^2(\Omega)$ are the embedding operators defined in \cref{sec:2}.

We define $B_j : H^1_0(\Omega) \to H^{-1}(\Omega)$ by $\langle B_j v,w\rangle := \langle
\psi_j \nabla v,\nabla w\rangle$ for all $v,w\in H^1_0(\Omega)$ so that $A(\pmb y)-A_s(\pmb y) =
\sum_{j\ge s+1} y_j\, B_j$, and define also 
\begin{align*}
T_s(\pmb y) := \sum_{j\ge s+1} y_j\, A_s^{-1}(\pmb y) B_j = A_s^{-1}(\pmb y)(A(\pmb y)-A_s(\pmb y))\,.
\end{align*}
Then for all $v \in H_0^1(\Omega)$ we can write using \cref{eq:Amin}
\begin{align*}
\|A_s^{-1}(\pmb y)B_jv\|_{H_0^1(\Omega)} \leq \frac{\|B_jv\|_{H^{-1}(\Omega)}}{a_{\min}} 
&= \frac{1}{a_{\min}} \sup_{w \in H_0^1(\Omega)} \frac{|\langle B_jv,w\rangle|}{\|w\|_{H_0^1(\Omega)}}\\
&= \frac{1}{a_{\min}} \sup_{w \in H_0^1(\Omega)} \frac{|\langle \psi_j \nabla v, \nabla w\rangle |}{\|w\|_{H_0^1(\Omega)}}\\
&\leq \frac{\|v\|_{H_0^1(\Omega)}}{a_{\min}} \|\psi_j\|_{L^{\infty}(\Omega)} = \|v\|_{H_0^1(\Omega)} b_j\,.
\end{align*}
We conclude that
\begin{align}
\sup_{\pmb y \in \Xi} \|A_s^{-1}(\pmb y)B_j\|_{\mathcal L(H_0^1(\Omega))} \leq b_j\,, \label{apple}
\end{align}
and consequently 
\begin{align*}
\sup_{\pmb y \in \Xi} \|T_s(\pmb y)\|_{\mathcal L(H_0^1(\Omega))} \leq \frac{1}{2} \sum_{j\geq s+1} b_j\,.
\end{align*}

Let $s^*$ be such that $\sum_{j \geq s^*+1} b_j \leq \frac{1}{2}$, implying that $\sup_{\pmb y \in \Xi}\|T_s(\pmb y)\|_{\mathcal L(H_0^1(\Omega))} \leq \frac{1}{4}$. Then for all $s \geq s^*$, by the bounded invertibility of $A(\pmb y)$ and $A_s(\pmb y)$ for all $\pmb y \in \Xi$, we can write (omitting $\pmb y$ in the following) the inverse of $A$ in terms of the Neumann series, as

\[
  A^{-1} \,=\, (I+T_s)^{-1} A_s^{-1} \,=\, \sum_{k\ge 0} (-T_s)^k A_s^{-1}\,.
\]
So
\[
  A^{-1} - A_s^{-1} \,=\, \sum_{k\ge 1} (-T_s)^k A_s^{-1}.
\]
Now let $E:=E_1E_2$ be the embedding operator of $H_0^1(\Omega)$ in $H^{-1}(\Omega)$. Then we can write the adjoint solution $q$ as $q = A^{-1}E_1(E_2u-u_0)$ and 
we write
\begin{align*}
  q - q_s
  &\,=\, A^{-1}E_1(E_2u-u_0) - A_s^{-1}E_1(E_2u_s - u_0) \\
  &\,=\, A^{-1}E_1(E_2u_s-u_0) - A_s^{-1}E_1(E_2u_s - u_0) + A^{-1}E(u-u_s) \\
  &\,=\, (A^{-1}-A_s^{-1})E_1(E_2u_s-u_0) + A^{-1}E(A^{-1}-A_s^{-1})E_1z \\
  &\,=\, \sum_{k\ge 1} (-T_s)^k A_s^{-1}E_1(E_2u_s-u_0) + \sum_{\ell\ge 0} (-T_s)^\ell A_s^{-1} E\bigg(\sum_{k\ge 1} (-T_s)^k A_s^{-1} E_1z\bigg) \\
  &\,=\, \sum_{k\ge 1} (-1)^k\, T_s^k q_s + \sum_{\ell\ge 0} \sum_{k\ge 1} (-1)^{\ell+k}\, T_s^\ell A_s^{-1} E\, T_s^k u_s\,,
\end{align*}
Thus
\begin{align*}
  \int_\Xi (q - q_s) \,\rd\pmb y
  &\,=\, \sum_{k\ge 1} (-1)^k \underbrace{\int_\Xi T_s^k q_s \,\rd\pmb y}_{=:\;{\rm Integral}_1}
  + \sum_{\ell\ge 0} \sum_{k\ge 1} (-1)^{\ell+k} \underbrace{\int_\Xi T_s^\ell A_s^{-1} E\, T_s^k u_s \,\rd\pmb y}_{=:\;{\rm Integral}_2},
\end{align*}
giving
\begin{align*}
  &\bigg\|\int_\Xi (q - q_s) \,\rd\bsy\bigg\|_{L^2}
  \,\le\, c_2\,\bigg\|\int_\Xi (q - q_s) \,\rd\bsy\bigg\|_{H^1_0} \\
  &\qquad\,\le\, c_2 \bigg(
  \underbrace{\sum_{k\ge 1} \bigg\|\int_\Xi T_s^k q_s\,\rd\bsy\bigg\|_{H^1_0}}_{=:\,{\rm Term}_1}
  +
  \underbrace{\sum_{\ell\ge 0} \sum_{k\ge 1} \bigg\|\int_\Xi T_s^\ell A_s^{-1} E\, T_s^k u_s\,\rd\bsy\bigg\|_{H^1_0}}_{=:\,{\rm Term}_2}
  \bigg).
\end{align*}

Noting that $B_j$ is independent of~$\bsy$, we write
\[
  T_s^k \,=\, \bigg(\sum_{j\ge s+1} y_j\, A_s^{-1} B_j\bigg)^k
  \,=\, \sum_{\bseta\in\{s+1:\infty\}^k} \prod_{i=1}^k (y_{\eta_i}\, A_s^{-1} B_{\eta_i})\,,
\]
where we use the shorthand notation $\{s+1:\infty\}^k = \{s+1,s+2,\ldots,\infty\}^k$. 
\bigskip

First we consider ${\rm Integral}_1$. We have
\begin{align*}
  &{\rm Integral}_1 \,=\, \int_\Xi \sum_{\bseta\in\{s+1:\infty\}^k} \bigg(\prod_{i=1}^k (y_{\eta_i}\, A_s^{-1} B_{\eta_i})\bigg)
  q_s\,\rd\bsy \\
  &\,=\, \sum_{\bseta\in\{s+1:\infty\}^k}
  \bigg(\int_{\Xi_{s+}} \prod_{i=1}^k y_{\eta_i}\,\rd\bsy_{\{s+1:\infty\}}\bigg)
  \bigg(\int_{\Xi_s} \bigg(\prod_{i=1}^k (A_s^{-1} B_{\eta_i})\bigg)
  q_s\,\rd\bsy_{\{1:s\}} \bigg),
\end{align*}
where we were able to separate the integrals for $\bsy_{\{1:s\}}$ and $\bsy_{\{s+1:\infty\}} := (y_j)_{j\ge s+1}$ and $\Xi_{s+} := \{(y_j)_{j\geq s+1}\,:\,y_j \in \left[-\frac{1}{2},\frac{1}{2}\right],\, j\geq s+1\}$, an essential step of this proof. The integral over $\bsy_{\{s+1:\infty\}}$ is nonnegative due to the simple yet
crucial observation that
\begin{align} \label{eq:simple}
\int_{-\frac{1}{2}}^{\frac{1}{2}} y_j^{n}\,\rd y_j \,=\,
  \begin{cases}
  0  & \mbox{if $n$ is odd}, \\
  \frac{1}{2^{n}(n+1)}  & \mbox{if $n$ is even}.
  \end{cases}
\end{align}
Using \cref{coro:2.4} and \cref{apple}, the $H^1_0$-norm of the integral over $\bsy_{\{1:s\}}$ can be estimated by
\begin{align*}
  &\sup_{\bsy_{\{1:s\}}\in \Xi_s} \bigg\|\bigg(\prod_{i=1}^k (A_s^{-1} B_{\eta_i})\bigg)
   q_s\bigg\|_{H^1_0} \le\, C_1\, \prod_{i=1}^k b_{\eta_i}\,,
\end{align*}
with $C_1 := C_q\,\big(\|z\|_{L^2(\Omega)}+\|u_0\|_{L^2(\Omega)}\big)$.
Hence we obtain
\begin{align*}
  {\rm Term}_1
  &\,\le\, C_1\,
  \sum_{k\ge 1}
  \sum_{\bseta\in\{s+1:\infty\}^k}
  \bigg(\int_{\Xi_{s+}} \prod_{i=1}^k y_{\eta_i}\,\rd\bsy_{\{s+1:\infty\}}\bigg)\,
  \prod_{i=1}^k b_{\eta_i} \\
  &\,=\, C_1\,
  \sum_{k\ge 1}
\int_{\Xi_{s+}}
  \sum_{\bseta\in\{s+1:\infty\}^k}
  \bigg(\prod_{i = 1}^k y_{\eta_i} b_{\eta_i}\bigg)\,
  \mathrm d\bsy_{\{s+1:\infty\}} \\
  &\,=\, C_1\,
  \sum_{k\ge 1}
  \sum_{\substack{|\bsnu|=k \\ \nu_j=0\;\forall j\le s}} \binom{k}{\bsnu}
  \bigg(\prod_{j\ge s+1}\int_{-\frac{1}{2}}^{\frac{1}{2}} y_j^{\nu_j}\,\rd y_j\bigg)\,
  \prod_{j\ge s+1} b_j^{\nu_j} \\
  &\,\leq\, C_1\,
  \sum_{k\ge 2\;{\rm even}}
  \sum_{\substack{|\bsnu|=k \\ \nu_j=0\;\forall j\le s \\ \nu_j\;{\rm even}\;\forall j\ge s+1}} \binom{k}{\bsnu}
  \prod_{j\ge s+1} b_j^{\nu_j},
\end{align*}
where the second equality follows from the multinomial theorem with
$\bsnu\in\mathcal D$ a multi-index,
$\binom{k}{\bsnu} = k!/(\prod_{j\ge 1}\nu_j!)$, while the last inequality
follows from~\eqref{eq:simple}.

Now we split the sum into a sum over $k \geq k^*$ and the initial terms $2\leq k< k^*$, and estimate
\begin{align*}
  {\rm Term}_1
  &\,\le\, C_1\, \sum_{k\ge k^*\;{\rm even}} \bigg(\sum_{j\ge s+1} b_j\bigg)^k
  + C_1\, \sum_{2\le k < k^* \;{\rm even}} k! \bigg( \prod_{j\ge s+1} \Big(1 + \sum_{t=2}^k b_j^t\Big) -1\bigg) \\
  &\,\le\, C_1\, \cdot\, \frac{4}{3} \bigg(\sum_{j\ge s+1} b_j\bigg)^{k^*}
  + C_1\, k^*\, k^*!\, \cdot\, 2(\mathrm e-1) \sum_{j\ge s+1} b_j^2.
\end{align*}
The estimate for the sum over $k \geq k^*$ follows from the multinomial theorem, the geometric series formula, and that $\sum_{j\geq s+1} b_j \leq \frac{1}{2}$ for $s\geq s^*$. For the sum over $2 \leq k < k^*$ we use the fact that for $s \geq s^*$ we have $b_j \leq \frac{1}{2}$ for all $j \geq s+1$ and $\sum_{j \geq s+1} b_j^2 \leq \sum_{j\geq s+1} b_j \leq \frac{1}{2}$, and thus
\begin{align*}
 \prod_{j\ge s+1} \Big(1 + \sum_{t=2}^k b_j^t\Big) -1 &= \prod_{j \geq s+1} \Big( 1 + b_j^2 \frac{1-b_j^{k-1}}{1-b_j} \Big) - 1 \leq 
\prod_{j\geq s+1} \left(1+\frac{b_j^2}{1-b_j}\right) -1\\ &\leq \prod_{j \geq s+1} \left(1+2 b_j^2\right)-1 \leq \exp\left(2 \sum_{j\geq s+1} b_j^2\right) -1 \leq 2(\mathrm e-1)\sum_{j \geq s+1} b_j^2\,,
\end{align*}
since $\mathrm e^r-1 \leq  r(\mathrm e-1)$ for $r \in [0,1]$.

\bigskip
Next we estimate ${\rm Integral}_2$ in a similar way. We have
\begin{align*}
  {\rm Integral}_2 &= 
  \int_\Xi \! \sum_{\bsmu\in\{s+1:\infty\}^\ell} \!
  \bigg(\prod_{i=1}^\ell (y_{\mu_i}\, A_s^{-1} B_{\mu_i})\bigg)
   A_s^{-1} E \bigg(\sum_{\bseta\in\{s+1:\infty\}^k}
  \bigg(\prod_{i=1}^k (y_{\eta_i}\, A_s^{-1} B_{\eta_i})\bigg)\bigg) u_s\,\rd\bsy\\
  \,&=\,
  \sum_{\bsmu\in\{s+1:\infty\}^\ell} \sum_{\bseta\in\{s+1:\infty\}^k}
  \bigg(
  \int_{\Xi_{s+}} \bigg(\prod_{i=1}^\ell y_{\mu_i}\bigg)
  \bigg(\prod_{i=1}^k y_{\eta_i}\bigg)\,\rd\bsy_{\{s+1:\infty\}}\bigg) \\
  &\qquad\qquad\qquad\qquad \cdot
  \bigg(\int_{\Xi_s}
  \bigg(\prod_{i=1}^\ell (A_s^{-1} B_{\mu_i})\bigg)
   A_s^{-1} E \bigg(\prod_{i=1}^k (A_s^{-1} B_{\eta_i})\bigg) u_s\,\rd\bsy_{\{1:s\}}\bigg),
\end{align*}
where we again separated the integrals for $\bsy_{\{s+1:\infty\}}$ and $\bsy_{\{1:s\}}$. With \cref{eq:2.5} and \cref{apple} we have
\begin{align*}
  &\sup_{\bsy_{\{1:s\}}\in \Xi_s} \bigg\|\bigg(\prod_{i=1}^\ell (A_s^{-1} B_{\mu_i})\bigg)
   A_s^{-1}E \bigg(\prod_{i=1}^k (A_s^{-1} B_{\eta_i})\bigg) u_s\bigg\|_{H^1_0} \le\, C_2\, \bigg(\prod_{i=1}^\ell b_{\mu_i}\bigg)\bigg(\prod_{i=1}^k b_{\eta_i}\bigg)\,,
\end{align*}
with $C_2 := c_1c_2\frac{c_1\|z\|_{L^2(\Omega)}}{a_{\min}^2}$.
Hence we obtain
\begin{align*}
{\rm Term}_2&\,\le\, C_2 
  \sum_{\ell\ge 0} \sum_{k\ge 1}
  \sum_{\bsmu\in\{s+1:\infty\}^\ell} \sum_{\bseta\in\{s+1:\infty\}^k}
  \bigg(
  \int_{\Xi_{s+}} \bigg(\prod_{i=1}^\ell y_{\mu_i}\bigg)
  \bigg(\prod_{i=1}^k y_{\eta_i}\bigg)\rd\bsy_{\{s+1:\infty\}}\bigg)\\
  &\qquad\qquad\qquad\qquad\qquad \cdot\bigg(\prod_{i=1}^\ell b_{\mu_i}\bigg)\bigg(\prod_{i=1}^k b_{\eta_i}\bigg) \\
 &\,=\, C_2\,
  \sum_{\ell\ge 0} \sum_{k\ge 1}
  \int_{\Xi_{s+}}   
  \sum_{\bsmu\in\{s+1:\infty\}^\ell} \sum_{\bseta\in\{s+1:\infty\}^k}
  \bigg(\prod_{i=1}^\ell y_{\mu_i} b_{\mu_i} \bigg)
  \bigg(\prod_{i=1}^k y_{\eta_i} b_{\eta_i} \bigg)\,\rd\bsy_{\{s+1:\infty\}} \\
  &\,=\, C_2\,
  \sum_{\ell\ge 0} \sum_{k\ge 1}
  \sum_{\substack{|\bsm|=\ell \\ m_j=0\;\forall j\le s}}
  \sum_{\substack{|\bsnu|=k \\ \nu_j=0\;\forall j\le s}} \binom{\ell}{\bsm} \binom{k}{\bsnu}
  \bigg(\prod_{j\ge s+1} \int_{-\frac{1}{2}}^{\frac{1}{2}} y_j^{m_j+\nu_j}\,\rd y_j\bigg)
  \prod_{j\ge s+1} b_j^{m_j+\nu_j} \\
  &\,\leq\, C_2\,
  \sum_{\ell\ge 0} \sum_{k\ge 1}
  \sum_{\substack{|\bsm|=\ell \\ m_j=0\;\forall j\le s}}
  \sum_{\substack{|\bsnu|=k \\\nu_j=0\;\forall j\le s \\ m_j+\nu_j\;{\rm even}\;\forall j\ge s+1}}
  \binom{\ell}{\bsm} \binom{k}{\bsnu}
  \prod_{j\ge s+1} b_j^{m_j+\nu_j}.
\end{align*}
Now we split the sums and estimate them in a similar way to the sums in $\rm {Term}_1$:

\begin{align*}
  {\rm Term}_2
  &\,\le\, C_2\, \sum_{\ell\ge 0} \sum_{\substack{k\ge 1 \\ \ell\ge \ell^* \;{\rm or}\; k\ge k^* \;{\rm or\;both}}} 
  \bigg(\sum_{j\ge s+1} b_j\bigg)^{\ell+k} \\
  &\qquad + C_2\,
  \sum_{0\le \ell < \ell^*} \sum_{\substack{1\le k < k^* \\ \ell+k\;{\rm even}}} \ell!\, k!
\sum_{\substack{|\pmb w|=\ell+k \\ w_j=0\;\forall j\le s \\ w_j\;{\rm even}\;\forall j\ge s+1}}
  \bigg(\prod_{j\ge s+1} b_j^{w_j}\bigg)
  \bigg(\sum_{|\pmb m|=\ell} \sum_{\substack{|\pmb \nu|=k \\ \pmb m+\pmb \nu = \pmb w}} 1\bigg)\,.
\end{align*}

For $s \geq s^*$ and denoting $P:= \sum_{j\geq s+1} b_j\,\leq\, \frac{1}{2}$ we can simplify the first part as
\begin{align*}
\sum_{\ell\ge 0} \sum_{\substack{k\ge 1 \\ \ell\ge \ell^* \;{\rm or}\; k\ge k^* \;{\rm or\;both}}}P^{\ell+k}
&=\sum_{\ell\geq \ell^\ast}\sum_{k=1}^{k^\ast-1}P^{\ell+k}+\sum_{\ell=0}^{\ell^\ast-1}\sum_{k\geq k^\ast}P^{\ell+k}+\sum_{\ell\geq \ell^\ast}\sum_{k\geq k^\ast}P^{\ell+k}\\
&=\sum_{\ell\geq \ell^\ast}\frac{P^{\ell+1}(1-P^{k^\ast-1})}{1-P} +\sum_{\ell=0}^{\ell^\ast-1} \frac{P^{\ell+k^\ast}}{1-P}+\sum_{\ell\geq \ell^\ast} \frac{P^{\ell+k^\ast}}{1-P}\\
&=\frac{P^{\ell^\ast+1}}{(1-P)^2}+ \frac{P^{k^\ast}}{(1-P)^2}- \frac{P^{\ell^\ast+k^\ast}}{(1-P)^2} \leq 8P^{\min(k^*,\ell^*)}\,.
\end{align*}

For a given multi-index $\pmb w$ satisfying $|\pmb w|=\ell+k$, $w_j=0$ for
all $j\le s$, and $w_j$ is even for all $j\ge s+1$, we need to count the
number of pairs of multi-indices $\pmb m$ and $\pmb \nu$ such that
$|\pmb m|=\ell$, $|\pmb \nu|=k$, and $\pmb m+\pmb \nu=\pmb w$ to estimate the second part. Clearly we have $w_j\le
\ell+k$, $m_j\le\ell$, $\nu_j\le k$ for all $j$. So the number of ways to
write any component $w_j$ as a sum $m_j + \nu_j$ is at most
$\min(w_j+1,\ell+1,k+1) \le \min(\ell,k)+1$. Moreover, since all $w_j$ are
even, there are at most $(\ell+k)/2$ nonzero components of $w_j$.
Therefore
\begin{align*}
  \sum_{|\pmb m|=\ell} \sum_{\substack{|\pmb \nu|=k \\ \pmb m+\pmb \nu = \pmb w}} 1
  \,\le\, [\min(\ell,k)+1]^{(\ell+k)/2}.
\end{align*}

Thus we obtain
\begin{align*}
  {\rm Term}_2
  &\,\le\, C_2\, \sum_{\ell\ge 0} \sum_{\substack{k\ge 1 \\ \ell\ge \ell^* \;{\rm or}\; k\ge k^* \;{\rm or\;both}}}
  \bigg(\sum_{j\ge s+1} b_j\bigg)^{\ell+k} \\
  &\qquad + C_2\, \sum_{0\le \ell < \ell^*} \sum_{\substack{1\le k < k^* \\ \ell+k\;{\rm even}}}
  \ell!\, k!\,[\min(\ell,k)+1]^{(\ell+k)/2}
  \bigg( \prod_{j\ge s+1} \Big(1 + \sum_{t=2}^{\ell+k} b_j^t\Big) -1\bigg) \\
  &\,\le\, C_2\,\cdot\, 8 \left(\sum_{j\ge s+1} b_j\right)^{\min(\ell^*,k^*)} \\
  &\qquad + C_2\, \ell^*\,k^*\, \ell^*! \,k^*!
  [\min(\ell^*,k^*)+1]^{(\ell^*+k^*)/2}\,\cdot\,2(\mathrm{e}-1)\,\sum_{j\ge s+1} b_j^2.
\end{align*}

From \cite[Theorem 5.1]{Kuo2012QMCFEM} we know that
\begin{align*}
\sum_{j \geq s+1} b_j \leq \min{\left( \frac{1}{\frac{1}{p}-1},1 \right)} \left(\sum_{j \geq 1} b_j^p \right)^{\frac{1}{p}} s^{-\left(\frac{1}{p}-1\right)}\,.
\end{align*}
Further there holds
$
b_j^p \leq \frac{1}{j} \sum_{l = 1}^j b_l^p \leq \frac{1}{j} \sum_{l = 1}^{\infty} b_l^p
$
and therefore
\begin{align*}
\sum_{j\geq s+1} b_j^2 = \sum_{j\geq s+1} (b_j^p)^{\frac{2}{p}} \leq \sum_{j \geq s+1} \left(\frac{1}{j} \sum_{l =1}^{\infty} b_l^p\right)^{\frac{2}{p}} = \left(\sum_{j\geq s+1} j^{-\frac{2}{p}} \right) \left(\sum_{l = 1}^{\infty} b_l^p \right)^{\frac{2}{p}}\,,
\end{align*}
and
\begin{align*}
\sum_{j \geq s+1} j^{-\frac{2}{p}} \leq \int_s^{\infty} t^{-\frac{2}{p}}\, \mathrm dt = \frac{1}{\frac{2}{p}-1}  s^{-\frac{2}{p}+1} \leq s^{-\left(\frac{2}{p}-1\right)}\,,
\end{align*}
for $0<p< 1$ as desired.

To balance the two terms within $\rm Term_1$ and the two terms within $\rm Term_2$, we now choose $\ell^* = k^* =  \lceil(2-p)/(1-p)\rceil$. We see that $\rm Term_1$ and $\rm Term_2$ are then of the order $s^{-(2/p-1)}$, which is what we aimed to prove.
\end{proof}

\begin{remark}\label{remarkUTrunc}
By the same analysis as for $\rm Term_1$ in the proof of \cref{theorem:Truncationerror} with $q_s$ replaced by $u_s$, we get the following
\begin{align*}
\left\|\int_{\Xi}\left(u(\cdot,\pmb y,z)-u_s(\cdot,\pmb y,z)\right)\,\mathrm d\pmb y\,\right\|_{L^2(\Omega)}
&\leq c_2 \sum_{k\geq1} \left\|\int_{\Xi} T_s^k u_s(\cdot,\pmb y,z)\,\mathrm d\pmb y\right\|_{H_0^1(\Omega)} \\
&\leq \tilde{C}_1 \Bigg(\frac{4}{3} \Big( \sum_{j\geq s+1} b_j\Big)^{k^*} + k^*\, k^*!\,\cdot\,2(\mathrm e - 1)\sum_{j \geq s+1} b_j^2\Bigg)\\
&\leq \tilde{C}\,s^{-\left(\frac{2}{p}-1\right)}\,,
\end{align*}
where $\tilde{C}_1 = c_2 \frac{\|z\|_{H^{-1}(\Omega)}}{a_{\min}}$, and some constant $\tilde{C}>0$ independent of $s$.
\end{remark}


\subsection{FE discretization}\label{section:FE discretization}

We follow \cite{Kuo2012QMCFEM} and in order to obtain convergence rates of the finite element solutions we make the following additional assumptions
\begin{align}
\text{$\Omega \subset \mathbb R^d$ is convex bounded polyhedron with plane faces}\,,\label{eq:A6}\\
\bar a \in W^{1,\infty}(\Omega)\,, \quad \sum_{j\geq 1} \|\psi_j\|_{W^{1,\infty}(\Omega)} < \infty\,,\label{eq:A4}
\end{align}
where $\|v\|_{W^{1,\infty}(\Omega)} := \max\{\|v\|_{L^{\infty}(\Omega)}\,, \|\nabla v\|_{L^{\infty}(\Omega)}\}$.
The assumption that the geometry of the computational domain $\Omega$ is approximated exactly by the FE mesh simplifies the forthcoming analysis, however, this assumption can substantially be relaxed. For example, standard results on FE analysis as, e.g., in \cite{Ciarlet} will imply corresponding results for domains $\Omega$ with curved boundaries.

In the following let $\{V_h\}_h$ denote a one-parameter family of subspaces $V_h \subset H_0^1(\Omega)$ of dimensions $M_h < \infty$, where $M_h$ is of exact order $h^{-d}$, with $d = 1,2,3$ denoting the spatial dimension.
We think of the spaces $V_h$ as spaces spanned by continuous, piecewise linear finite element basis functions on a sequence of regular, simplicial meshes in $\Omega$ obtained from an initial, regular triangulation of $\Omega$ by recursive, uniform bisection of simplices. Then it is well known (see details, e.g., in \cite{Gilbarg,Kuo2012QMCFEM}) that for functions $v \in H_0^1(\Omega) \cap H^2(\Omega)$ there exists a constant $C>0$, such that as $h \to 0$
\begin{align}\label{FEabschaetzung}
\inf_{v_h \in V_h} \|v-v_h\|_{H_0^1(\Omega)} \leq C\, h\, \|v\|_{H_0^1(\Omega) \cap H^2(\Omega)}\,,
\end{align}
where $\|v\|_{H_0^1(\Omega) \cap H^2(\Omega)} := ( \|v\|_{L^2(\Omega)}^2 + \|\Delta v\|_{L^2(\Omega)}^2)^{1/2}$. Note that we need the higher regularity in order to derive the asymptotic convergence rate as $h \to 0$. 
For any $\pmb y \in \Xi$ and every $z \in L^2(\Omega)$, we define the parametric finite element approximations $u_h(\cdot,\pmb y,z) \in V_h$ and $q_h(\cdot,\pmb y,z) \in V_h$ by
\begin{align}
b(\pmb y;u_h(\cdot,\pmb y,z),v_h) = \langle z, v_h\rangle \quad \forall v_h \in V_h\,,\label{eq:FEapprox2}
\end{align}
and then
\begin{align}
b(\pmb y;q_h(\cdot,\pmb y,z),w_h) =\langle u_h(\cdot,\pmb y,z)-u_0,w_h\rangle \quad \forall w_h \in V_h\,,\label{eq:FEapprox3}
\end{align}
where $b(\pmb y;\cdot,\cdot)$ is the parametric bilinear form \cref{eq:2.2}. In particular the FE approximation \cref{eq:FEapprox2} and \cref{eq:FEapprox3} are defined pointwise with respect to $\pmb y \in \Xi$ so that the application of a QMC rule to the FE approximation is well defined.
To stress the dependence on $s$ for truncated $\pmb y = (y_1,\ldots,y_s,0,0,\ldots) \in \Xi$ we write $u_{s,h}$  and $q_{s,h}$ instead of $u_{h}$ and $q_h$ in \cref{eq:FEapprox2}  and \cref{eq:FEapprox3}.

\begin{theorem}[Finite element discretization error] \label{theorem:FEapprox}
Under assumptions \cref{eq:A6} and \cref{eq:A4}, for $z \in \mathcal Z$, there holds the asymptotic convergence estimate as $h\to 0$
\begin{align*}\label{eq:theoremFEapproxNitsche}
\sup_{\pmb y \in \Xi} \|q(\cdot,\pmb y,z)-q_h(\cdot,\pmb y,z)\|_{L^2(\Omega)} \leq C h^2 \left(\|z\|_{L^2(\Omega)} + \|u_0\|_{L^2(\Omega)} \right)\,,
\end{align*}
and
\begin{align*}\label{eq:EqualWeightSup}
\left\| \int_{\Xi} \left(q(\cdot,\pmb y,z) -  q_{h}(\cdot,\pmb y,z)\right)\, \mathrm d\pmb y  \right\|_{L^2(\Omega)} \leq Ch^2 \left(\|z\|_{L^2(\Omega)} + \|u_0\|_{L^2(\Omega)} \right)\,,
\end{align*}
where $C>0$ is independent of $h$, $z$ and $u_0$ and $\pmb y$.
\end{theorem}

For truncated $\pmb y = (y_1,\ldots,y_s,0,0,\ldots) \in \Xi$, the result of \cref{theorem:FEapprox} clearly holds with $q$ and $q_h$ replaced by $q_s$ and $q_{s,h}$ respectively.

\begin{proof}
Let $S_{\pmb y,h}$ be the self-adjoint solution operator defined analogously to \cref{def:solutionoperator}; which for every $\pmb y\in \Xi$ assigns to each function $f \in L^2(\Omega)$ the unique solution $g_h(\cdot,\pmb y) \in V_h \subset H_0^1(\Omega) \subset L^2(\Omega)$. In particular $S_{\pmb y,h}$ is the solution operator of the problem: find $g_h \in V_h$ such that $b(\pmb y;g_h,v_h) = \langle f,v_h\rangle$ $\forall v_h \in V_h$. Note that $S_{\pmb y,h}$ is a bounded and linear operator for given $\pmb y \in \Xi$.
For every $\pmb y \in \Xi$, we can thus estimate
\begin{align}
\|q(\cdot,\pmb y,z)-q_h(\cdot,\pmb y,z)\|_{L^2(\Omega)} &= \|S_{\pmb y}(u(\cdot,\pmb y,z)-u_0)-S_{\pmb y,h}(u_h(\cdot,\pmb y,z)-u_0)\|_{L^2(\Omega)}\notag\\ 
&\leq \|S_{\pmb y}(u(\cdot,\pmb y,z)-u_0)-S_{\pmb y,h}(u(\cdot,\pmb y,z)-u_0)\|_{L^2(\Omega)}\notag\\
&\quad+ \|S_{\pmb y,h}u(\cdot,\pmb y,z)-S_{\pmb y,h}u_h(\cdot,\pmb y,z)\|_{L^2(\Omega)}\notag\\ 
&\leq \|(S_{\pmb y}-S_{\pmb y,h})(u(\cdot,\pmb y,z)-u_0)\|_{L^2(\Omega)}\notag\\
&\quad+ \frac{c_1c_2}{a_{\min}} \|u(\cdot,\pmb y,z)-u_h(\cdot,\pmb y,z)\|_{L^2(\Omega)}\,.\label{eq:AUB2}
\end{align}
The last step is true because \cref{eq:2.5} holds for all $v \in H_0^1(\Omega)$ and therefore it holds in particular for $u_h \in V_h \subset H_0^1(\Omega)$. Hence we can bound $\|S_{\pmb y,h}\|_{\mathcal L(L^2(\Omega))} \leq \frac{c_1c_2}{a_{\min}}$.
We can now apply the Aubin--Nitsche duality argument (see, e.g., \cite{Gilbarg}) to bound \cref{eq:AUB2}:
for $w\in L^2(\Omega)$ it holds that
\begin{align}\label{eq:AubinNitscheTrick1}
\|w\|_{L^2(\Omega)} = \sup_{g \in L^2(\Omega) \setminus \{0\}} \frac{\langle g,w\rangle}{\|g\|_{L^2(\Omega)}}\,.
\end{align}
From \cref{eq:parametricweakproblem} and \cref{eq:FEapprox2} follows the Galerkin orthogonality: $b(\pmb y;u(\cdot,\pmb y,z)-u_h(\cdot,\pmb y,z), v_h) = 0$ for all $v_h \in V_h$. Further we define $u_g(\cdot,\pmb y)$ for every $\pmb y \in \Xi$ as the unique solution of the problem: find $u_g(\cdot,\pmb y) \in H_0^1(\Omega)$ such that
\begin{align*}
b(\pmb y;u_g(\cdot,\pmb y),w) = \langle g,w\rangle \quad \forall w \in H_0^1(\Omega)\,,
\end{align*}
which leads together with the choice $w:= u - u_h$ and the Galerkin orthogonality of the FE discretization to
\begin{align*}
\langle g, u(\cdot,\pmb y,z) - u_h(\cdot,\pmb y,z) \rangle &= b(\pmb y;u_g(\cdot,\pmb y),u(\cdot,\pmb y,z) - u_h(\cdot,\pmb y,z) )\\
&= b(\pmb y; u_g(\cdot, \pmb y) - v_h, u(\cdot,\pmb y,z) - u_h(\cdot,\pmb y,z) )\\
&\leq a_{\max} \|u_g(\cdot,\pmb y) - v_h\|_{H_0^1(\Omega)} \|u(\cdot,\pmb y,z) - u_h(\cdot,\pmb y,z)\|_{H_0^1(\Omega)}\,.
\end{align*}
With \cref{eq:AubinNitscheTrick1} we get for every $\pmb y \in \Xi$ that
\begin{align*}
\|u(\cdot,\pmb y,z) - u_h(\cdot,\pmb y,z)\|_{L^2(\Omega)}& = \sup_{g \in L^2(\Omega) \setminus \{0\}} \frac{\langle g,  u(\cdot,\pmb y,z) - u_h(\cdot,\pmb y,z) \rangle}{\|g\|_{L^2(\Omega)}}&\hfill\\ 
&\!\!\!\!\!\!\!\!\!\!\!\!\!\!\!\!\!\!\!\!\!\!\!\!\!\!\!\!\!\!\!\!\!\!\!\!\!\!\!\!\!\!\!\!\!\!\leq a_{\max} \|u(\cdot,\pmb y,z) - u_h(\cdot,\pmb y,z)\|_{H_0^1(\Omega)} \sup_{g \in L^2(\Omega) \setminus \{0\}} \left\{ \inf_{v_h \in V} \frac{\|u_g(\cdot,\pmb y) - v_h\|_{H_0^1(\Omega)} }{\|g\|_{L^2(\Omega)}} \right\}\,.
\end{align*}
Now from \cref{FEabschaetzung} we infer for every $\pmb y \in \Xi$ that
\begin{align*}
\inf_{v_h \in V} \|u_g(\cdot,\pmb y) - v_h\|_{H_0^1(\Omega)} \leq C_3\,h\, \|u_g(\cdot,\pmb y) \|_{H_0^1(\Omega) \cap H^2(\Omega)}
\leq C_4 C_3 \,h\, \|g\|_{L^2(\Omega)}\,,
\end{align*}
where $C_3$ is the constant in \cref{FEabschaetzung}. The last step follows from \cite[Theorem 4.1]{Kuo2012QMCFEM} with $t=1$, and $C_4$ is the constant in that theorem.
For every $\pmb y \in \Xi$, we further obtain with C\'{e}a's lemma, \cref{FEabschaetzung} and \cite[Theorem 4.1]{Kuo2012QMCFEM}
\begin{align*}
\|u(\cdot,\pmb y,z) - u_h(\cdot,\pmb y,z)\|_{H_0^1(\Omega)} 
&\leq \frac{a_{\max}}{a_{\min}} \inf_{v_h \in V} \|u(\cdot,\pmb y,z) - v_h\|_{H_0^1(\Omega)}\\
&\leq \frac{a_{\max}}{a_{\min}} C_3 \,h\, \|u(\cdot,\pmb y,z)\|_{H_0^1(\Omega) \cap H^2(\Omega)}\\
&\leq \frac{a_{\max}}{a_{\min}}  C_4 C_3 \,h\, \|z\|_{L^2(\Omega)}\,.
\end{align*}
Thus for every $\pmb y \in \Xi$ it holds that
\begin{align}
\|u(\cdot,\pmb y,z)-u_h(\cdot,\pmb y,z)\|_{L^2(\Omega)} \leq \frac{a_{\max}^2}{a_{\min}} C^2_4 C^2_3\, h^2\, \|z\|_{L^2(\Omega)}\,.\label{eq:combFE1}
\end{align}
By the same argument we get for every $\pmb y \in \Xi$ that
\begin{align}
\|(S_{\pmb y}-S_{\pmb y,h})(u(\cdot,\pmb y,z)-u_0)\|_{L^2(\Omega)} 
\leq \frac{a_{\max}^2}{a_{\min}} C^2_4 C^2_3\, h^2 \left(\frac{c_1c_2}{a_{\min}}\|z\|_{L^2(\Omega)} + \|u_0\|_{L^2(\Omega)}\right).\label{eq:combFE2}
\end{align}
Combining \cref{eq:combFE1} and \cref{eq:combFE2} in \cref{eq:AUB2} leads for every $\pmb y \in \Xi$ to
\begin{align*}
\|q(\cdot,\pmb y,z)-q_h(\cdot,\pmb y,z)\|_{L^2(\Omega)} 
&\leq \frac{a_{\max}^2}{a_{\min}} C^2_4 C^2_3 \,h^2\left( \frac{2\,c_1c_2}{a_{\min}} \|z\|_{L^2(\Omega)} + \|u_0\|_{L^2(\Omega)}\right) .
\end{align*}
The second result easily follows from the first result since
\begin{align*}
\left\| \int_{\Xi} \left(q(\cdot,\pmb y,z) -  q_{h}(\cdot,\pmb y,z)\right)\, \mathrm d\pmb y\right\|_{L^2(\Omega)}^2 \leq \int_{\Xi} \| q(\cdot,\pmb y,z)-q_{h}(\cdot,\pmb y,z)\|_{L^2(\Omega)}^2\, \mathrm d\pmb y\,.
\end{align*}
\end{proof}


\subsection{Regularity of the adjoint solution}
\label{subsection54}

In the subsequent QMC error analysis we shall require bounds on the mixed first partial derivatives of the parametric solution $u$ as well as bounds on the mixed first partial derivatives of the adjoint parametric solution $q$.
For the solution $u(\cdot,\pmb y,z)$ of the state equation \cref{eq:parametricweakproblem} we know the following result.
\begin{lemma}\label{lemma:Derivative}
For every $z \in H^{-1}(\Omega)$, every $\pmb y \in \Xi$ and every $\pmb \nu \in \mathcal D$
we have
\begin{align*}
\|(\partial^{\pmb \nu} u)(\cdot,\pmb y,z) \|_{H_0^1(\Omega)} := \| \nabla (\partial^{\pmb \nu} u)(\cdot,\pmb y,z) \|_{L^2(\Omega)} \leq |\pmb \nu|!\, \pmb b^{\pmb \nu} \frac{\|z\|_{H^{-1}(\Omega)}}{a_{\min}}\,.
\end{align*}
\end{lemma}
This lemma can be found, e.g., in \cite{CohenDeVoreSchwab}. 

In contrast to the parametric weak problem \cref{eq:parametricweakproblem}, the right-hand side of the adjoint parametric weak problem  \cref{eq:adjointparametricweakproblem} depends on $\pmb y \in \Xi$.
In particular the problem is of the following form: for every $\pmb y \in \Xi$, find $q(\cdot,\pmb y,z) \in H_0^1(\Omega)$ such that
\begin{equation}\label{eq:4.11}
\int_{\Omega} a(\pmb x,\pmb y) \nabla q(\pmb x,\pmb y,z) \cdot \nabla v(\pmb x)\, \mathrm d\pmb x = \int_{\Omega} \tilde f(\pmb x,\pmb y,z) v(\pmb x)\, \mathrm d\pmb x\,, \quad v \in H_0^1(\Omega)\,,
\end{equation}
where the right-hand side $\tilde f(\pmb x,\pmb y,z) := u(\pmb x,\pmb y,z)-u_0(\pmb x)$ now also depends on $z \in L^2(\Omega)$ and $\pmb y \in \Xi$. 
\cref{lemma:adjointDerivative} below gives a bound for the mixed derivatives of the solution $q(\cdot,\pmb y,z) \in H_0^1(\Omega)$ of \cref{eq:4.11}. Similar regularity results to the following can be found in \cite{KunothSchwab} (uniform case) and \cite{ChenGhattas} (log-normal case) for problems with stochastic controls $z$, depending on $\pmb y$. In particular, in the unconstrained case $\mathcal Z = L^2(\Omega)$ the KKT-system \cref{eq:KKT} reduces to an affine parametric linear saddle point operator and the theory, e.g., from \cite{KunothSchwab, Schwab} can be applied.

\begin{lemma}\label{lemma:adjointDerivative}For every $z \in L^2(\Omega)$, every $\pmb y \in \Xi$ and every $\pmb \nu \in \mathcal D$, we have for the corresponding adjoint state $q(\cdot,\pmb y,z)$ that
\begin{align*}
\|(\partial^{\pmb \nu} q)(\cdot,\pmb y,z) \|_{H_0^1(\Omega)} \leq (|\pmb \nu| + 1)!\, \pmb b^{\pmb \nu}\, C_q\,(\|z\|_{L^2(\Omega)} + \|u_0\|_{L^2(\Omega)})\,,
\end{align*}
where $C_q$ is defined in \cref{coro}.
\end{lemma}

\begin{proof}
The case $\pmb \nu = \pmb 0$ is given by the \emph{a priori} bound \cref{coro:2.4}. Now consider $\pmb \nu \neq \pmb 0$.
Applying the mixed derivative operator $\partial^{\pmb \nu}$ to \cref{eq:4.11} and using the Leibniz product rule, we obtain the identity
\begin{align}
\int_{\Omega} \left( \sum_{\pmb m\leq \pmb \nu} \begin{pmatrix} \pmb \nu \\ \pmb m \end{pmatrix} (\partial^{\pmb \nu}a)(\pmb x,\pmb y) \nabla (\partial^{\pmb \nu- \pmb m}q)(\pmb x,\pmb y,z) \cdot \nabla v(\pmb x) \right)\, \mathrm d\pmb x\\ = \int_{\Omega} (\partial^{\pmb \nu} \tilde f)(\pmb x,\pmb y,z)\ v(\pmb x)\, \mathrm d\pmb x \quad \forall v \in H_0^1(\Omega)\notag\,,
\end{align}
where by $\pmb m \leq \pmb \nu$ we mean $m_j \leq \nu_j$ for all $j$ and $\binom{\pmb \nu}{\pmb m} := \prod_{j\geq1} \binom{\nu_j}{m_j}$.
Due to the linear dependence of $a(\pmb x,\pmb y)$ on the parameters $\pmb y$, the partial derivative $\partial^{\pmb \nu}$ of $a$ with respect to $\pmb y$ satisfies
\begin{align*}
(\partial^{\pmb m} a)(\pmb x,\pmb y) = \begin{cases}
a(\pmb x,\pmb y) & \text{if } \pmb m = \pmb 0\,,\\
\psi_j(\pmb x) & \text{if } \pmb m = \pmb e_j\,, \\
0 & \text{else}\,.
\end{cases}
\end{align*}
Setting $v = (\partial^{\pmb \nu}q)(\cdot,\pmb y,z)$ and separating out the $\pmb m= \pmb 0$ term, we obtain
\begin{align*}
\int_{\Omega} a |\nabla (\partial^{\pmb \nu} q)(\pmb x,\pmb y,z)|^2\, \mathrm d\pmb x = &-\sum_{j \in \text{supp}(\pmb \nu)} \nu_j \int_{\Omega} \psi_j(\pmb x) \nabla (\partial^{\pmb \nu- \pmb e_j}q)(\pmb x,\pmb y,z) \cdot \nabla(\partial^{\pmb \nu}q)(\pmb x,\pmb y,z)\, \mathrm d\pmb x\\ 
&+ \int_{\Omega} (\partial^{\pmb \nu} \tilde f)(\pmb x,\pmb y,z) (\partial^{\pmb \nu}q)(\pmb x, \pmb y,z)\, \mathrm d\pmb x\,,
\end{align*}
which yields
\begin{align*}
a_{\min} \| (\partial^{\pmb \nu} q )(\cdot,\pmb y,z) \|^2_{H_0^1(\Omega)} &\leq \sum_{j \geq 1} \nu_j \|\psi_j\|_{L^{\infty}(\Omega)} \|(\partial^{\pmb \nu- \pmb e_j} q)(\cdot,\pmb y,z)\|_{H_0^1(\Omega)} \| (\partial^{\pmb \nu}q)(\cdot,\pmb y,z)\|_{H_0^1(\Omega)}\\
&\quad+ \|(\partial^{\pmb \nu} \tilde f)(\cdot,\pmb y,z) \|_{H^{-1}(\Omega)}  \|(\partial^{\pmb \nu}q)(\cdot,\pmb y,z)\|_{H_0^1(\Omega)} \\
&= \sum_{j \geq 1} \nu_j \|\psi_j\|_{L^{\infty}(\Omega)} \|(\partial^{\pmb \nu- \pmb e_j} q)(\cdot,\pmb y,z)\|_{H_0^1(\Omega)} \| (\partial^{\pmb \nu}q)(\cdot,\pmb y,z)\|_{H_0^1(\Omega)}\\
&\quad+ \|(\partial^{\pmb \nu} \tilde f)(\cdot,\pmb y,z) \|_{H^{-1}(\Omega)}  \| (\partial^{\pmb \nu}q)(\cdot,\pmb y,z)\|_{H_0^1(\Omega)} \,,
\end{align*}
and hence
\begin{align*}
\| (\partial^{\pmb \nu} q )(\cdot,\pmb y,z) \|_{H_0^1(\Omega)} &\leq \sum_{j \geq 1} \nu_j b_j \|(\partial^{\pmb \nu-\pmb e_j} q)(\cdot,\pmb y,z)\|_{H_0^1(\Omega)} + \frac{ \|(\partial^{\pmb \nu} \tilde f)(\cdot,\pmb y,z) \|_{H^{-1}(\Omega)} }{a_{\min} }\,. 
\end{align*}
With $\tilde f(\cdot,\pmb y,z) = u(\cdot,\pmb y,z)-u_0(\cdot)$ this reduces to
\begin{align}
\| (\partial^{\pmb \nu} q )(\cdot,\pmb y,z) \|_{H_0^1(\Omega)} &\leq \sum_{j \geq 1} \nu_j b_j \|(\partial^{\pmb \nu- \pmb e_j} q)(\cdot,\pmb y,z)\|_{H_0^1(\Omega)} + \frac{ \|(\partial^{\pmb \nu} u)(\cdot,\pmb y,z) \|_{H^{-1}(\Omega)} }{a_{\min} }\,. \label{eq:4.14}
\end{align}
With \cref{lemma:Derivative} we get
\begin{align*}
\|(\partial^{\pmb \nu} u)(\cdot,\pmb y,z) \|_{H^{-1}(\Omega)} \leq c_1c_2 \|(\partial^{\pmb \nu}u)(\cdot,\pmb y,z)\|_{H_0^1(\Omega)} \leq c_1c_2\, |\pmb \nu|!\, \pmb b^{\pmb \nu} \frac{\|z\|_{H^{-1}(\Omega)}}{a_{\min}}\,,
\end{align*}
where $c_1,c_2>0$ are embedding constants, see \cref{c1,c2}. 
Then \cref{eq:4.14} becomes, for $\pmb \nu \neq \pmb 0$,
\begin{align*}
\| (\partial^{\pmb \nu} q )(\cdot,\pmb y,z) \|_{H_0^1(\Omega)} &\leq \sum_{j \geq 1} \nu_j b_j \|(\partial^{\pmb \nu- \pmb e_j} q)(\cdot,\pmb y,z)\|_{H_0^1(\Omega)} + c_1c_2\, |\pmb \nu|!\, \pmb b^{\pmb \nu} \frac{\|z\|_{H^{-1}(\Omega)}}{a^2_{\min}} \,.
\end{align*}

Now we apply \cite[Lemma 9.1]{Kuo2016ApplicationOQ} to obtain the final bound. For this to work we need the above recursion to hold also for the case $\pmb \nu = \pmb 0$, which is not true when we compare it with the \emph{a priori} bound \cref{coro:2.4}. We therefore enlarge the constants so that the recursion becomes
\begin{align*}
\| (\partial^{\pmb \nu} q )(\cdot,\pmb y,z) \|_{H_0^1(\Omega)} &\leq \sum_{j \geq 1} \nu_j b_j \|(\partial^{\pmb \nu- \pmb e_j} q)(\cdot,\pmb y,z)\|_{H_0^1(\Omega)} + |\pmb \nu|!\, \pmb b^{\pmb \nu}\, C_q\,(\|z\|_{L^2(\Omega)} + \|u_0\|_{L^2(\Omega)}) \,,
\end{align*}
which by \cite[Lemma 9.1]{Kuo2016ApplicationOQ} gives
\begin{align*}
\| (\partial^{\pmb \nu} q )(\cdot,\pmb y,z) \|_{H_0^1(\Omega)} &\leq \sum_{\pmb m \leq\pmb  \nu} \begin{pmatrix} \pmb \nu \\ \pmb m \end{pmatrix} |\pmb m|!\ \pmb b^{\pmb m}\, |\pmb \nu -\pmb m|!\ \pmb b^{\pmb \nu -\pmb m}\, C_q\,(\|z\|_{L^2(\Omega)} + \|u_0\|_{L^2(\Omega)})\\
&= \pmb b^{\pmb \nu}\, C_q\,(\|z\|_{L^2(\Omega)} + \|u_0\|_{L^2(\Omega)}) \sum_{\pmb m \leq \nu} \begin{pmatrix}\pmb  \nu \\ \pmb m \end{pmatrix} |\pmb m|!\ |\pmb \nu -\pmb m|!\\
&= \pmb b^{\pmb \nu}\, C_q\,(\|z\|_{L^2(\Omega)} + \|u_0\|_{L^2(\Omega)})\, (|\pmb \nu| + 1)!\,,
\end{align*}
where the last equality from \cite[equation 9.4]{Kuo2016ApplicationOQ} and $C_q$ is defined in \cref{coro}.
\end{proof}


\subsection{QMC integration error}\label{subsection:QMC}

In this section we review QMC integration over the $s$-dimensional unit cube $\Xi_s = \left[-\frac{1}{2},\frac{1}{2}\right]^s$ centered at the origin, for finite and fixed $s$. An $n$-point QMC approximation is an equal-weight rule of the form
\begin{align*}
\int_{\left[-\frac{1}{2},\frac{1}{2}\right]^s} F(\pmb y_{\{1:s\}})\, \mathrm d\pmb y_{\{1:s\}} \approx \frac{1}{n} \sum_{i=1}^{n} F(\pmb y^{(i)})\,,
\end{align*}
with carefully chosen points $\pmb y^{(1)},\ldots,\pmb y^{(n)} \in \Xi_{s}$.
We shall assume that for each $s\geq 1$ the integrand $F$ belongs to a weighted unanchored Sobolev space $\mathcal W_{\pmb \gamma,s}$, which is a Hilbert space containing functions defined over the unit cube $\left[-\frac{1}{2},\frac{1}{2}\right]^s$, with square integrable mixed first derivatives, with norm given by
\begin{align*}
\|F\|_{\mathcal W_{\pmb \gamma,s}} := \left( \sum_{{\mathfrak u }\subseteq \{1:s\}} \gamma_{{\mathfrak u}}^{-1}  \int_{\left[-\frac{1}{2},\frac{1}{2}\right]^{|{\mathfrak u}|}}\ \left( \int_{\left[-\frac{1}{2},\frac{1}{2}\right]^{s-|{\mathfrak u|}}} \frac{\partial^{|{\mathfrak u}|}F}{\partial \pmb y_{{\mathfrak u}}}(\pmb y_{{\mathfrak u}};\pmb y_{\{1:s\}\setminus{\mathfrak u}})\, \mathrm d\pmb y_{\{1:s\}\setminus{\mathfrak u}} \right)^2\, \mathrm d\pmb y_{{\mathfrak u}}  \right)^{\frac{1}{2}}\,,\label{sobolevnorm}
\end{align*}
where we denote by $\frac{\partial^{|{\mathfrak u}|}F}{\partial \pmb y_{{\mathfrak u}}}$ the mixed first derivative with respect to the active variables $y_j$ with $j \in {\mathfrak u} \subset \mathbb N$ and $\pmb y_{\{1:s\}\setminus {\mathfrak u}}$ denotes the inactive variables $y_j$ with $j \notin {\mathfrak u}$.

We assume there is a weight parameter $ \gamma_{{\mathfrak u}} \geq 0$ associated with each group of variables $\pmb y_{{\mathfrak u}} = (y_j)_{j \in {\mathfrak u}}$ with indices belonging to the set ${\mathfrak u}$. We require that if $\gamma_{{\mathfrak u}} = 0$ then the corresponding integral of the mixed first derivative is also zero and we follow the conventions that $0/0 = 0$, $\gamma_{\emptyset} = 1$ and by $\pmb \gamma$ we denote the set of all weights.
See \cref{section:optimalweights} for the precise choice of weights.

In this work we focus on shifted rank-1 lattice rules, which are QMC rules with quadrature points given by
\begin{align*}
\pmb y^{(i)} = \text{frac}\left(\frac{i\pmb z}{n}+\pmb{\Delta} \right) - \left(\frac{1}{2},\ldots,\frac{1}{2}\right)\,, \quad i = 1,\ldots,n\,,
\end{align*}
where $\pmb z \in \mathbb N^s$ is known as the generating vector, $\pmb{\Delta} \in [0,1]^s$ is the shift and frac$(\cdot)$ means to take the fractional part of each component in the vector. The subtraction of $ \left(\frac{1}{2},\ldots,\frac{1}{2}\right)$ ensures the translation from the usual unit cube $[0,1]^s$ to $\left[-\frac{1}{2},\frac{1}{2}\right]^s$. 

\begin{theorem}[QMC quadrature error] \label{theorem:QMCerror}
For every $\pmb y_{\{1:s\}} \in \Xi_s$ let $q_{s,h}(\cdot,\pmb y_{\{1:s\}},z) \in V_h \subset H_0^1(\Omega)$ denote the dimensionally truncated adjoint FE solution corresponding to a control $z \in \mathcal Z$. Then for ${\mathfrak u} \subset \mathbb N$, $s,m \in \mathbb N$ with $n = 2^m$ and weights $\pmb \gamma = (\gamma_{ {\mathfrak u}})$, a randomly shifted lattice rule with $n$ points in $s$ dimensions can be constructed by a CBC algorithm such that the root-mean-square $L^2$-error $e_{s,h,n}$ for approximating the finite-dimensional integral $\int_{\Xi_s} q_{s,h}(\cdot,\pmb y_{\{1:s\}},z)\ \mathrm d\pmb y_{\{1:s\}}$ satisfies, for all $\lambda \in (\frac{1}{2},1]$,
\begin{align*}
e_{s,h,n} :=& \sqrt{\mathbb E_{\pmb{\Delta}} \left[ \left\|\int_{\left[-\frac{1}{2},\frac{1}{2}\right]^s} q_{s,h}(\cdot,\pmb y_{\{1:s\}},z)\, \mathrm d\pmb y_{\{1:s\}} - \frac{1}{n} \sum_{i = 1}^{n} q_{s,h}(\cdot,\pmb y^{(i)},z) \right\|_{L^2(\Omega)}^2 \right]}\\
\leq& \sqrt{c_2}\, C_{\pmb \gamma,s}(\lambda) \left(\frac{2}{n}\right)^{\frac{1}{2\lambda}}\, C_q\,(\|z\|_{L^2(\Omega)} + \|u_0\|_{L^2(\Omega)})\,,
\end{align*}
where
\begin{align*}
C_{\pmb \gamma,s}(\lambda) := \left( \sum_{{\mathfrak u} \subseteq \{1:s\}} \gamma_{{\mathfrak u}}^{\lambda} \left(\frac{2\zeta(2\lambda)}{(2\pi^2)^{\lambda}}\right)^{|{\mathfrak u}|} \right)^{\frac{1}{2\lambda}} \left(\sum_{{\mathfrak u} \subseteq \{1:s\}} \frac{((|{\mathfrak u}|+1)!)^2}{\gamma_{{\mathfrak u}}} \prod_{j \in {\mathfrak u}} b_j^2 \right)^{\frac{1}{2}}\,,
\end{align*}
where $\mathbb E_{\pmb{\Delta}}[\cdot]$ denotes the expectation with respect to the random shift which is uniformly distributed over $[0,1]^s$, and $\zeta(x) := \sum_{k=1}^{\infty} k^{-x}$ is the Riemann zeta function for $x>1$. Further, the $b_j$ are defined in \cref{b_j} and $C_q$ is defined in \cref{coro}.
\end{theorem}

\begin{proof}
We have
\begin{align*}
(e_{s,h,n})^2 &= \mathbb E_{\pmb{\Delta}} \left[ \int_{\Omega} \left| \int_{\Xi_s} q_{s,h}(\pmb x,\pmb y_{\{1:s\}},z)\, \mathrm d\pmb y_{\{1:s\}} - \frac{1}{n} \sum_{i = 1}^n q_{s,h}(\pmb x,\pmb y^{(i)},z) \right|^2 \mathrm d\pmb x \right]\\
&= \int_{\Omega} \mathbb E_{\pmb{\Delta}} \left[ \left| \int_{\Xi_s} q_{s,h}(\pmb x,\pmb y_{\{1:s\}},z)\, \mathrm d\pmb y_{\{1:s\}} - \frac{1}{n} \sum_{i = 1}^n q_{s,h}(\pmb x,\pmb y^{(i)},z) \right|^2 \right]\,\mathrm d\pmb x \\
&\leq \left( \sum_{\emptyset \neq \mathfrak u \subseteq \{1:s\}} \gamma_{{\mathfrak u}}^{\lambda} \left(\frac{2\zeta(2\lambda)}{(2\pi^2)^{\lambda}}\right)^{|{\mathfrak u|}} \right)^{\frac{1}{\lambda}} \left(\frac{2}{n}\right)^{\frac{1}{\lambda}}\int_{\Omega} \|q_{s,h}(\pmb x,\cdot,z)\|_{\mathcal W_{\pmb \gamma,s}}^2\, \mathrm d\pmb x\,,
\end{align*}
where we used Fubini's theorem in the second equality and \cite[Theorem 2.1]{Kuo2012QMCFEM} to obtain the inequality. Now from the definition of the $\mathcal W_{\pmb \gamma,s}$-norm (see \cref{sobolevnorm}), we have
\begin{align*}
&\int_{\Omega} \|q_{s,h}(\pmb x,\cdot,z)\|_{\mathcal W_{\pmb \gamma,s}}^2\, \mathrm d\pmb x\\
&= \int_{\Omega} \sum_{{\mathfrak u} \subseteq \{1:s\}} \frac{1}{\gamma_{{\mathfrak u}}} \int_{\left[-\frac{1}{2},\frac{1}{2}\right]^{|{\mathfrak u}|}} \left( \int_{\left[-\frac{1}{2},\frac{1}{2}\right]^{s-|{\mathfrak u}|}}  \frac{\partial^{|\mathfrak u|}q_{s,h}}{\partial\pmb  y_{{\mathfrak u}}}(\pmb x,(\pmb y_{{\mathfrak u}};\pmb y_{\{1:s\} \setminus {\mathfrak u}}),z)\, \mathrm d\pmb y_{\{1:s\}\setminus {\mathfrak u}} \right)^2\, \mathrm d\pmb y_{{\mathfrak u}}\, \mathrm d\pmb x\\
&\leq \sum_{{\mathfrak u} \subseteq \{1:s\}} \frac{1}{\gamma_{{\mathfrak u}}} \int_{\Omega} \int_{\left[-\frac{1}{2},\frac{1}{2}\right]^{|{\mathfrak u}|}} \int_{\left[-\frac{1}{2},\frac{1}{2}\right]^{s-|{\mathfrak u}|}}  \left( \frac{\partial^{|{\mathfrak u}|}q_{s,h}}{\partial\pmb  y_{{\mathfrak u}}}(\pmb x,(\pmb y_{{\mathfrak u}};\pmb y_{\{1:s\} \setminus {\mathfrak u}}),z) \right)^2\, \mathrm d\pmb y_{\{1:s\}\setminus {\mathfrak u}}\, \mathrm d\pmb y_{{\mathfrak u}}\, \mathrm d\pmb x\\
&= \sum_{{\mathfrak u} \subseteq \{1:s\}} \frac{1}{\gamma_{{\mathfrak u}}} \int_{\Xi_{s}}  \left\| \frac{\partial^{|{\mathfrak u}|}q_{s,h}}{\partial \pmb y_{{\mathfrak u}}}(\cdot,\pmb y_{\{1:s\}},z) \right\|^2_{L^2(\Omega)}\, \mathrm d\pmb y_{\{1:s\}}\\
&\leq c_2 \sum_{{\mathfrak u} \subseteq \{1:s\}} \frac{1}{\gamma_{{\mathfrak u}}} \int_{\Xi_{s}}  \left\| \frac{\partial^{|{\mathfrak u}|}q_{s,h}}{\partial \pmb y_{{\mathfrak u}}}(\cdot,\pmb y_{\{1:s\}},z) \right\|^2_{H_0^1(\Omega)}\, \mathrm d\pmb y_{\{1:s\}}\\
&\leq c_2  \sum_{{\mathfrak u} \subseteq \{1:s\}} \frac{1}{\gamma_{{\mathfrak u}}} \left( (|{\mathfrak u}|+1)!\, C_q\, (\|z\|_{L^2(\Omega)} + \|u_0\|_{L^2(\Omega)}) \prod_{j \in {\mathfrak u}} b_j \right)^2\,,
\end{align*}
where the first inequality uses the Cauchy--Schwarz inequality and the last inequality uses \cref{lemma:adjointDerivative}.
\end{proof}

\begin{remark}\label{remarkU}
From the proof of \cref{theorem:QMCerror} it can easily be seen that we can get an analogous result to \cref{theorem:QMCerror} by replacing $q_{s,h}$ with $u_{s,h}$ and using \cref{lemma:Derivative} instead of \cref{lemma:adjointDerivative} in the last step of the proof.
\end{remark}


\subsection{Optimal weights}\label{section:optimalweights}
In the following we choose weights $\gamma_{{\mathfrak u}}$ so that $C_{\pmb \gamma,s}(\lambda)$ in \cref{theorem:QMCerror} is bounded independently of $s$. To do so we follow and adjust the discussion in \cite{KuoNuyens2018} and therefore assume
\begin{align}\label{assump:psummability}
\sum_{j \geq 1} \|\psi_j\|_{L^{\infty}(\Omega)}^{p} < \infty\,,
\end{align}
for $p \in (0,1)$.

For any $\lambda$, $C_{\pmb \gamma,s}(\lambda)$ is minimized with respect to the weights $\gamma_{{\mathfrak u}}$ by
\begin{align}\label{eq:weights}
\gamma_{{\mathfrak u}} = \Bigg((|{\mathfrak u}|+1)! \prod_{j \in {\mathfrak u}} \frac{b_j}{\big(\frac{2\zeta(2\lambda)}{(2\pi^2)^{\lambda}}\big)^{1/2}} \Bigg)^{2/(1+\lambda)}\,,
\end{align}
see also \cite[Lemma 6.2]{Kuo2012QMCFEM}.
We substitute \cref{eq:weights} into $C_{\pmb \gamma,s}(\lambda)$ and simplify the expression to
\begin{align}\label{rmserror2}
C_{\pmb \gamma,s}(\lambda) =  \left( \sum_{{\mathfrak u} \subseteq \{1:s\}} { \left((|{\mathfrak u}|+1)! \prod_{j \in {\mathfrak u}} b_j \left(\frac{2\zeta(2\lambda)}{(2\pi^2)^{\lambda}}\right)^{1/(2\lambda)} \right)^{2\lambda/(1+\lambda)}} \right)^{(1+\lambda)/(2\lambda)} \,.
\end{align}
Next derive a condition on $\lambda$ for which \cref{rmserror2} is bounded independently of $s$.
Let $\phi := b_j \left(\frac{2\zeta(2\lambda)}{(2\pi^2)^{\lambda}}\right)^{1/(2\lambda)}$ and $k := \frac{2\lambda}{1+\lambda}$, then it holds that
\begin{align*}
\sum_{{\mathfrak u} \subseteq \{1:s\}} \left( (|{\mathfrak u}| + 1)! \prod_{j \in {\mathfrak u}} \phi_j \right)^k = \sum_{l=0}^s ((l+1)!)^k \sum_{\substack{{\mathfrak u} \subseteq \{1:s\}\\ |{\mathfrak u}| = l}} \prod_{j \in {\mathfrak u}} \phi_j^k \leq \sum_{l = 0}^s \frac{((l+1)!)^k}{l!} \left(\sum_{j=1}^s \phi_j^k\right)^l \,.
\end{align*}

With the ratio test we obtain, that the right-hand side is bounded independently of $s$ if $\sum_{j=1}^{\infty} \phi_j^k < \infty$ and $k <1$. We have $\sum_{j=1}^{\infty} \phi_j^k = \left(\frac{2\zeta(2\lambda)}{(2\pi^2)^{\lambda}}\right)^{1/(1+\lambda)} \sum_{j=1}^{\infty} b_j^k < \infty$ if $k \geq p$, where $p$ is the summability exponent in \cref{assump:psummability}.
Thus we require
\begin{align}\label{eq:constraintsonlambda}
p \leq \frac{2\lambda}{1+\lambda} < 1 \quad \Leftrightarrow \quad \frac{p}{2-p} \leq \lambda < 1\,.
\end{align}

Since the best rate of convergence is obtained for $\lambda$ as small as possible, combining \cref{eq:constraintsonlambda} with $\lambda \in \left(\frac{1}{2},1\right]$ yields

\begin{align}\label{eq:choicelambda}
\lambda = 
\begin{cases}
\frac{1}{2-2\delta} & \text{for all } \delta \in \left(0,\frac{1}{2}\right) \text{ if } p \in \left(0,\frac{2}{3} \right]\,,\\
\frac{p}{2-p} &  \hfill \text{ if } p \in \left(\frac{2}{3} ,1\right)\,.
\end{cases}
\end{align}

\begin{theorem}[Choice of the weights]\label{theorem:choiceofweights}
Under assumption \cref{assump:psummability}, the choice of $\lambda$ as in \cref{eq:choicelambda} together with the choice of the weights \cref{eq:weights} ensures that the bound on $e_{s,h,n}$ is finite independently of $s$.
(However, $C_{\pmb \gamma,s}\left(\frac{1}{2-2\delta}\right) \to \infty$ as $\delta \to 0$ and $C_{\pmb \gamma,s}\left(\frac{p}{2-p}\right) \to \infty$ as $p \to (2/3)^+$.) In consequence under assumption \cref{assump:psummability} and the same assumptions as in \cref{theorem:QMCerror}, the root-mean-square error in \cref{theorem:QMCerror} is of order
\begin{align}
\kappa(n) := 
\begin{cases}
n^{-(1-\delta)} & \text{for all } \delta \in \left(0,\frac{1}{2}\right)  \text{ if } p \in \left(0,\frac{2}{3} \right]\,,\\
n^{-(1/p - 1/2)} & \hfill  \text{ if } p \in \left(\frac{2}{3} ,1\right)\,.
\end{cases}\label{kappa}
\end{align}
\end{theorem}


\subsection{Combined error and convergence rates}
\label{subsection56}

Combining the results of the preceding subsections gives the following theorem.

\begin{theorem}[Combined error]\label{theorem:combinederror}
Let $z^*$ be the unique solution of \cref{eq:3.1} and $z^*_{s,h,n}$ the unique solution of \cref{eq:QMCFEobjective}. Then under the assumptions of \cref{theorem:Truncationerror}, \cref{theorem:FEapprox}, \cref{theorem:QMCerror} and \cref{theorem:choiceofweights}, we have
\begin{align*}\label{eq:finalresult1}
\sqrt{\mathbb E_{\pmb{\Delta}}[\|z^*-z^*_{s,h,n}\|_{L^2(\Omega)}^2]} \leq \frac{C}{\alpha} (\|z^*\|_{L^2(\Omega)} + \|u_0\|_{L^2(\Omega)})\left( s^{-\frac{2}{p}+1} + h^2 + \kappa(n) \right)\,,
\end{align*}
where $\kappa(n)$ is given in \cref{kappa}.
\end{theorem}

\begin{proof}
Squaring \cref{eq:convergence} and using the expansion \cref{eq:errorexpansion} we get by taking expectation with respect to the random shift $\pmb \Delta$
\begin{align*}
\mathbb E_{\pmb{\Delta}} \left[ \|z^*-z^*_{s,h,n}\|^2_{L^2(\Omega)} \right] 
&\leq  \frac{2}{\alpha^2} \left\| \int_{\Xi} \left(q(\cdot,\pmb y,z^*) - q_s(\cdot,\pmb y,z^*)\right)\, \mathrm d\pmb y \right\|_{L^2(\Omega)}^2\\ 
&+ \frac{2}{\alpha^2} \left\| \int_{\Xi_s} \left(q_s(\cdot,\pmb y_{\{1:s\}},z^*) - q_{s,h}(\cdot,\pmb y_{\{1:s\}},z^*)\right)\, \mathrm d\pmb y_{\{1:s\}} \right\|_{L^2(\Omega)}^2\\
&\!\!\!\!\!\!\!\!\!+  \frac{1}{\alpha^2} \mathbb E_{\pmb{\Delta}} \left[ \left\| \int_{\Xi_s} q_{s,h}(\cdot,\pmb y_{\{1:s\}},z^*)\, \mathrm d\pmb y_{\{1:s\}} - \frac{1}{n} \sum_{i = 1}^{n} q_{s,h}(\cdot,\pmb y^{(i)},z^*) \right\|_{L^2(\Omega)}^2\right]  \,.
\end{align*}
The result then immediately follows from \cref{theorem:Truncationerror}, \cref{theorem:FEapprox} and \cref{theorem:QMCerror}.
\end{proof}

Using the error bound for the control $z_{s,h,n}^*$ in \cref{theorem:combinederror} we obtain an error estimate for the state $u_{s,h}(\cdot,\pmb y,z_{s,h,n}^*)$ in the following corollary.

\begin{corollary}\label{finalcorollary}
Let $z^*$ be the unique solution of \cref{eq:3.1} and $z^*_{s,h,n}$ the unique solution of \cref{eq:QMCFEobjective}, then under the assumptions of \cref{theorem:combinederror} we have
\begin{align*}
&\sqrt{\mathbb E_{\pmb{\Delta}}\left[ \int_\Xi \|u(\cdot,\pmb y,z^*) - u_{s,h}(\cdot,\pmb y,z_{s,h,n}^*)\|^2_{L^2(\Omega)}\, \mathrm d\pmb y \right] }\\ &\quad \quad \leq C (\|z^*\|_{L^2(\Omega)} + \|u_0\|_{L^2(\Omega)})\left( s^{-\left(\frac{1}{p}-1\right)} + h^2 + \kappa(n) \right)\,,
\end{align*}
where $\kappa(n)$ is given in \cref{kappa}.
\end{corollary}

\begin{proof}
We observe that the error in $u_{s,h}(\cdot,\pmb y,z_{s,h,n}^*)$ compared to $u(\cdot,\pmb y,z^*)$ has three different sources, which can be estimated separately as follows
\begin{align*}
\|u(\cdot,\pmb y,z^*) - u(\cdot,\pmb y,z_{s,h,n}^*)\|_{L^2(\Omega)} 
&\leq \|u(\cdot,\pmb y, z^*) - u_s(\cdot,\pmb y,z^*)\|_{L^2(\Omega)}\\ 
&\qquad + \|u_s(\cdot,\pmb y,z^*) - u_{s,h}(\cdot,\pmb y,z^*)\|_{L^2(\Omega)}\\
&\qquad + \|u_{s,h}(\cdot,\pmb y,z^*) - u_{s,h}(\cdot,\pmb y,z_{s,h,n}^*)\|_{L^2(\Omega)}\\
&\leq \tilde{C}_1\, \|z^*\|_{L^2(\Omega)}\, s^{-\left(\frac{1}{p}-1\right)}
+ \tilde{C}_2\, \|z^*\|_{L^2(\Omega)}\, h^2\\
&\qquad + \frac{c_1\,c_2}{a_{\min}}\, \|z^* - z_{s,h,n}^*\|_{L^2(\Omega)}\,,
\end{align*}
where the bound for the first summand follows from \cite[Theorem 5.1]{Kuo2012QMCFEM}, the bound for the second summand follows from \cref{eq:combFE1} and the bound for the last summand can be obtained using \cref{theorem:theorem1}. Squaring both sides, taking expectation with respect to $\pmb y$ and with respect to the random shifts $\pmb{\Delta}$ and \cref{theorem:combinederror} gives the result.
\end{proof}
From the proof of \cref{finalcorollary} it can easily be seen that its statement remains true if the integral with respect to $\pmb y$ is replaced by the supremum over all $\pmb y \in \Xi$. In \cref{finalcorollary}, in contrast to \cref{theorem:Truncationerror}, we do not obtain the enhanced rate of convergence $s^{-(2/p-1)}$ for the dimension truncation. This is due to the difference in the order of application of the integral (with respect to $\pmb y$) and the $L^2(\Omega)$-norm.


\section{Numerical experiments}

\noindent We consider the coupled PDE system \cref{eq:adjointparametricweakproblem,eq:33b}
in the two-dimensional physical domain $\Omega=(0,1)^2$ equipped with the diffusion coefficient \cref{eq:1.9}.
We set $\bar{a}(\bsx)\equiv 1$ as the mean field and use the parametrized family of fluctuations
\begin{align}
\psi_{j}(\bsx)=\frac{1}{(k_{j}^2+\ell_{j}^2)^\vartheta}\sin(\pi k_{j}x_1)\sin(\pi \ell_{j}x_2)\quad\text{for}~\vartheta>1~\text{and}~j\in\mathbb{N},\label{eq:fluctuation}
\end{align}
where the sequence $(k_{j},\ell_{j})_{j\geq1}$ is an ordering of the elements of $\mathbb{N} \times \mathbb{N}$, so that the sequence $(\|\psi_j\|_{L^\infty(\Omega)})_{j\geq 1}$ is non-increasing. This implies that $\|\psi_{j}\|_{L^\infty(\Omega)}\sim j^{-\vartheta}$ as $j\to\infty$ by Weyl's asymptotic law for the spectrum of the Dirichlet Laplacian (cf.~\cite{weyl} as well as the examples in ~\cite{DKGS,Gantner}). 
We use a first order finite element solver to compute the solutions to the system~\cref{eq:adjointparametricweakproblem,eq:33b}
~numerically over an ensemble of regular hierarchical FE meshes $\{\mathcal{T}_h\}_h$ of the square domain $\Omega$, parametrized using the one-dimensional mesh widths $h\in\{2^{-k}: k\in\mathbb{N}\}$.

In the numerical experiments in \cref{subsectionNum1} to \cref{subsectionNum3}, we fix the source term $z(\bsx)=x_2$ and set $u_0(\bsx)=x_1^2-x_2^2$ for $\bsx=(x_1,x_2)\in \Omega$. The lattice QMC rule was generated in all experiments by using the fast CBC implementation of the QMC4PDE software~\cite{KN,Kuo2016ApplicationOQ}, with the weights chosen to appropriately accommodate the fluctuations~\eqref{eq:fluctuation} in accordance with \cref{theorem:choiceofweights}. In particular, we note that while all the lattice rules in the subsequent numerical examples were designed with the adjoint solution $q$ in mind, the same lattice rules have been used in the sequel to analyze the behavior of the state solution $u$ of \cref{eq:33b}
~as well. All computations were carried out on the Katana cluster at UNSW Sydney.

\subsection{Finite element error}
\label{subsectionNum1}
In this section, we assess the validity of the finite element error bounds given in \cref{theorem:FEapprox}.

Two numerical experiments were carried out:
\begin{itemize}
\item[(a)] The $L^2$ errors $\|u_s(\cdot,\bsy,z)-u_{s,h}(\cdot,\bsy,z)\|_{L^2(\Omega)}$ and $\|q_s(\cdot,\bsy,z)-q_{s,h}(\cdot,\bsy,z)\|_{L^2(\Omega)}$ of the FE solutions to the state and adjoint PDEs, respectively, were computed using the parameters $s=100$ and $h\in\{2^{-k}:k\in\{1,\ldots,9\}\}$ for a \emph{single} realization of the parametric vector $\bsy\in [-1/2,1/2]^{100}$ drawn from $U([-1/2,1/2]^{100})$.
\item[(b)] The terms  $\big\|\int_{\Xi_s}(u_s(\cdot,\bsy,z)-u_{s,h}(\cdot,\bsy,z))\,{\rm d}\bsy\big\|_{L^2(\Omega)}$ and $\big\|\int_{\Xi_s}(q_s(\cdot,\bsy,z)-\linebreak[4]q_{s,h}(\cdot,\bsy,z))\,{\rm d}\bsy\big\|_{L^2(\Omega)}$ were approximated by using a lattice rule with a single fixed random shift to evaluate the parametric integrals with dimensionality $s=100$, $n=2^{15}$ nodes and mesh width $h\in\{2^{-k}:k\in\{1,\ldots,6\}\}$.
\end{itemize}
The value $\vartheta=2.0$ was used in both experiments as the rate of decay for the fluctuations~\eqref{eq:fluctuation}. As the reference solutions $u_s$ and $q_s$, we used FE solutions computed using the mesh width $h=2^{-10}$ for experiment (a) and $h=2^{-7}$ for experiment (b). The $L^2$ errors were computed by interpolating the coarser FE solutions onto the grid corresponding to the reference solution. The numerical results are displayed in \cref{fig:femerr}. In the case of a single fixed vector $\bsy\in[-1/2,1/2]^{100}$, we obtain the rates $\mathcal{O}(h^{2.01688})$ and $\mathcal{O}(h^{2.00542})$ for the state and adjoint solutions, respectively. The corresponding rates averaged over $n=2^{15}$ lattice quadrature nodes are $\mathcal{O}(h^{2.04011})$ for the state PDE and $\mathcal{O}(h^{2.01617})$ for the adjoint PDE. In both cases, the observed rates adhere nicely with the theoretical rates given in \cref{theorem:FEapprox}.

\begin{figure}[!h]
\centering
\subfloat[]{{\includegraphics[height=.45\textwidth]{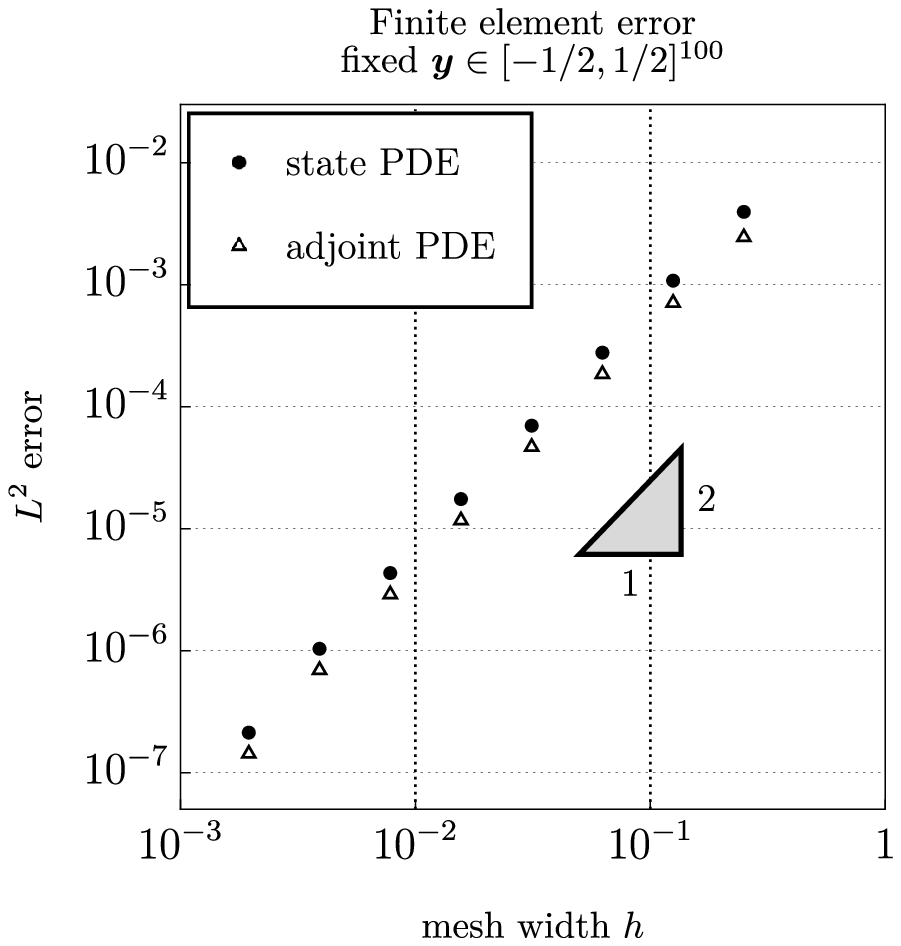}}}\quad\subfloat[]{{\includegraphics[height=.45\textwidth]{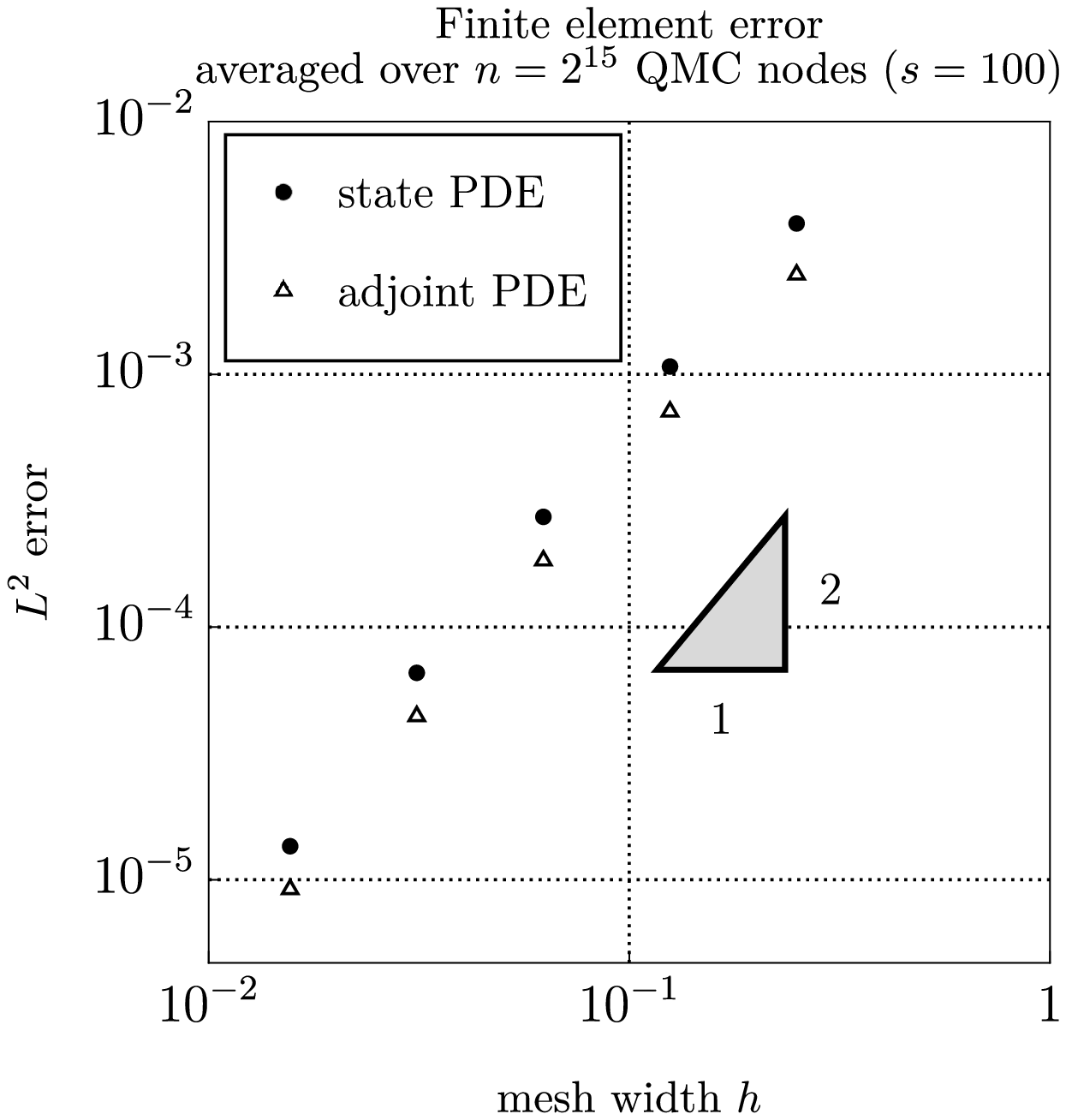}}}
\caption{The computed finite element errors displayed against the theoretical rates.}\label{fig:femerr}
\end{figure}

\subsection{Dimension truncation error}\label{subsectionNum2}
The dimension truncation error was estimated by approximating the quantities
$$
\bigg\|\int_{\Xi}(u(\cdot,\bsy,z)-u_s(\cdot,\bsy,z))\,{\rm d}\bsy\bigg\|_{L^2(\Omega)}\quad\text{and}\quad \bigg\|\int_{\Xi}(q(\cdot,\bsy,z)-q_s(\cdot,\bsy,z))\,{\rm d}\bsy\bigg\|_{L^2(\Omega)}
$$ 
using a lattice quadrature rule with $n=2^{15}$ nodes and a single fixed random shift to evaluate the parametric integrals. The coupled PDE system was discretized using the mesh width $h=2^{-5}$ and, as the reference solutions $u$ and $q$, we used the FE solutions corresponding to the parameters $s=2^{11}$ and $h=2^{-5}$. The obtained results are displayed in \cref{fig:dimtrunc} for the fluctuation operators $(\psi_{j})_{j\geq 1}$ corresponding to the decay rates $\vartheta\in\{1.5,2.0\}$ and dimensions $s\in\{2^k:k\in\{1,\ldots,9\}\}$. The numerical results are accompanied by the corresponding theoretical rates, which are $\mathcal{O}(s^{-2})$ for $\vartheta=1.5$ and $\mathcal{O}(s^{-3})$ for $\vartheta=2.0$ according to \cref{theorem:Truncationerror}. 

In all cases, we find that the observed rates tend  toward the expected rates as $s$ increases. In particular, by carrying out a least squares fit for the data points corresponding to the values $s\in\{2^5,\ldots,2^9\}$, the calculated dimension truncation error rate for the state PDE is $\mathcal{O}(s^{-2.00315})$ (corresponding to the decay rate $\vartheta=1.5$) and $\mathcal{O}(s^{-2.83015})$ (corresponding to the decay rate $\vartheta=2.0$). For the adjoint PDE, the corresponding rates are $\mathcal{O}(s^{-2.0065})$ and $\mathcal{O}(s^{-2.72987})$, respectively. The discrepancy between the obtained rate and the expected rate in the case of the decay parameter $\vartheta=2.0$ may be explained by two factors: the lattice quadrature error rate is at best linear, so the quadrature error is likely not completely eliminated with $n=2^{15}$ lattice quadrature points. Moreover, the rate obtained in \cref{theorem:Truncationerror} is sharp only for potentially high values of $s$. This phenomenon may also be observed in the slight curvature of the data presented in \cref{fig:dimtrunc}.

\begin{figure}[!h]
\centering
\subfloat{{\includegraphics[height=.45\textwidth]{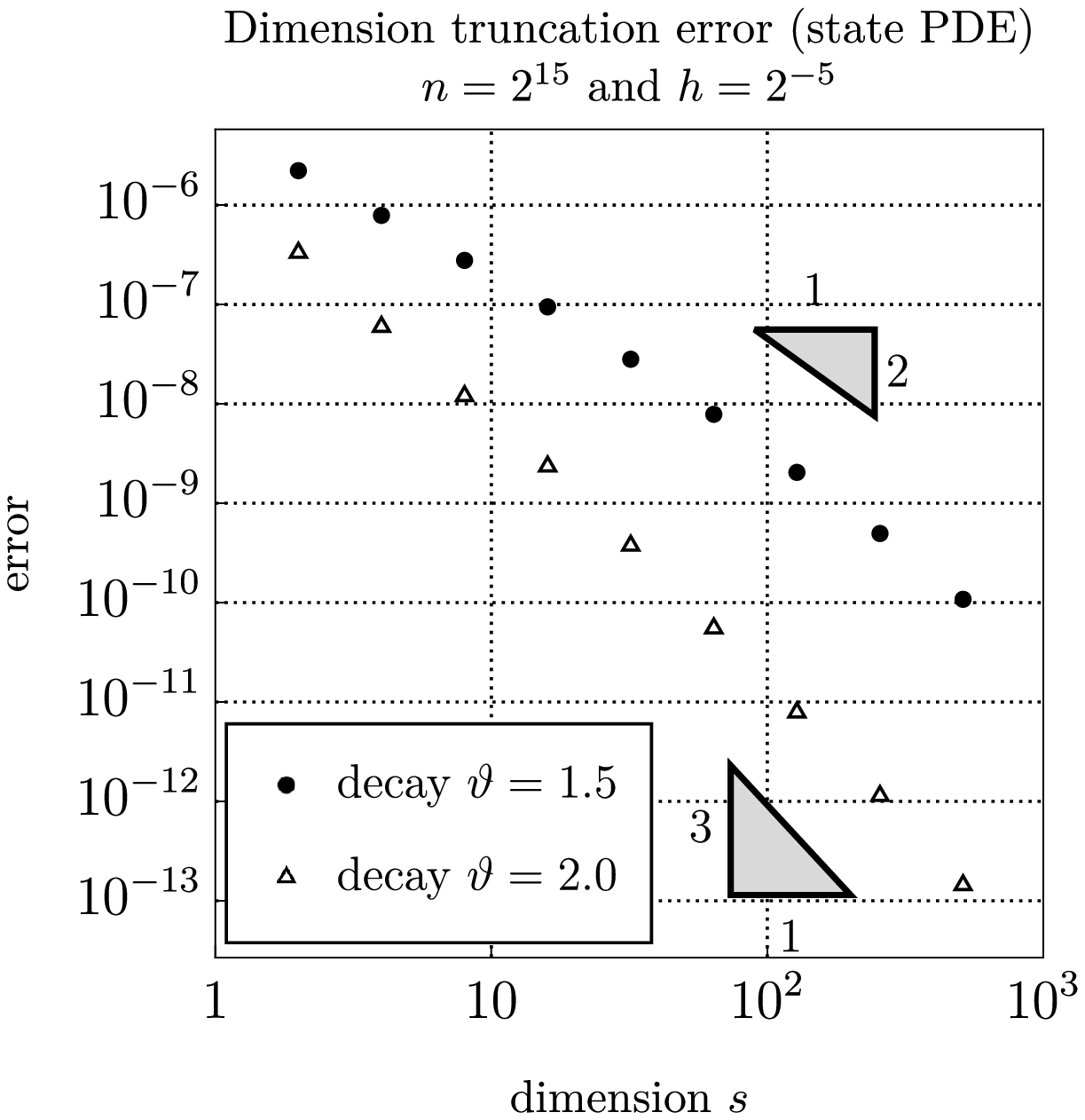}}}\qquad\subfloat{{\includegraphics[height=.45\textwidth]{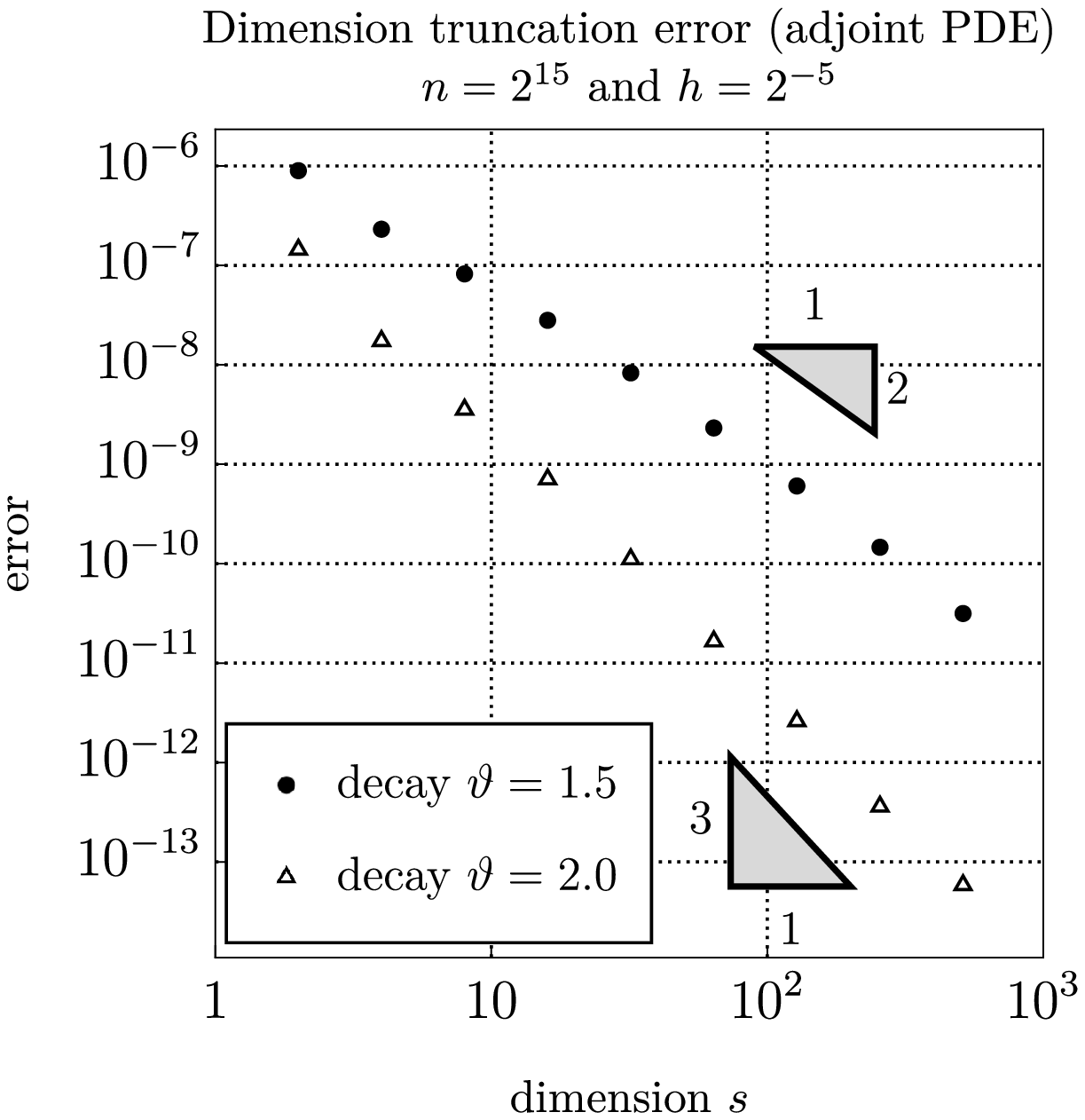}}}
\caption{The computed dimension truncation errors displayed against the expected rates.}\label{fig:dimtrunc}
\end{figure}

\subsection{QMC error}\label{subsectionNum3}
We assess the rate in \cref{theorem:QMCerror} by using the root-mean-square approximation
\begin{align*}
&\sqrt{\mathbb{E}_{\boldsymbol{\Delta}}\bigg\|\int_{\Xi_s}q_{s,h}(\cdot,\bsy_{\{1:s\}},z)\,{\rm d}\bsy_{\{1:s\}}-\frac{1}{n}\sum_{i=1}^nq_{s,h}(\cdot,\{\boldsymbol{t}^{(i)}+\boldsymbol{\Delta}\}-\tfrac{\boldsymbol{1}}{\boldsymbol{2}},z)\bigg\|_{L^2(\Omega)}^2}\\
&\approx \sqrt{\frac{1}{R(R-1)}\sum_{r=1}^R\big\|\overline{Q} - Q^{(r)}\big\|_{L^2(\Omega)}^2}\,,
\end{align*}
where
$Q^{(r)} := \frac{1}{n}\sum_{i=1}^n q_{s,h}(\cdot,\{\boldsymbol{t}^{(i)}+\boldsymbol{\Delta}^{(r)}\}-\tfrac{\boldsymbol{1}}{\boldsymbol{2}},z) 
$ and $
\overline{Q} = \frac{1}{R}\sum_{r=1}^R Q^{(r)}$,
for a randomly shifted lattice rule with $n=2^m$, $m\in\{7,\ldots,15\}$, lattice points $(\boldsymbol{t}^{(i)})_{i=1}^n$ in $[0,1]^s$ and $R=16$ random shifts $\boldsymbol{\Delta}^{(r)}$ drawn from $U([0,1]^s)$ with $s=100$. The FE  solutions were computed using the mesh width $h=2^{-6}$. The results are displayed in \cref{fig:qmc}. In both cases, the theoretical rate is $\mathcal{O}(n^{-1+\delta})$, $\delta>0$. For the decay rate $\vartheta=1.5$, we observe the rates $\mathcal{O}(n^{-0.984193})$ for the state PDE and $\mathcal{O}(n^{-0.987608})$ for the adjoint PDE. When the decay rate is $\vartheta=2.0$, we obtain the rates $\mathcal{O}(n^{-1.01080})$ and $\mathcal{O}(n^{-1.012258})$ for the state and adjoint PDE, respectively.

\begin{figure}[!h]
\centering
\subfloat{{\includegraphics[height=.45\textwidth]{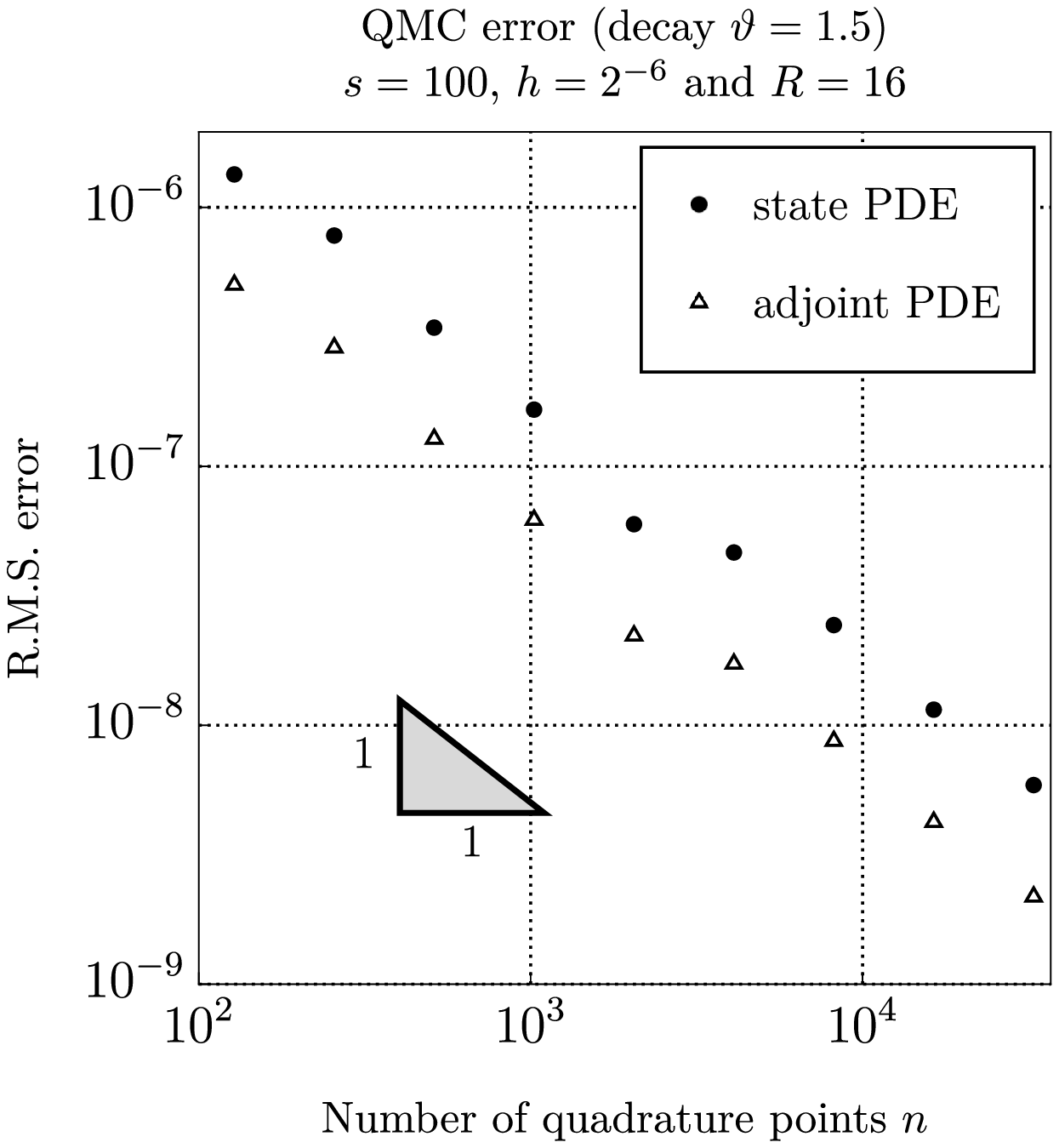}}}\qquad\subfloat{{\includegraphics[height=.45\textwidth]{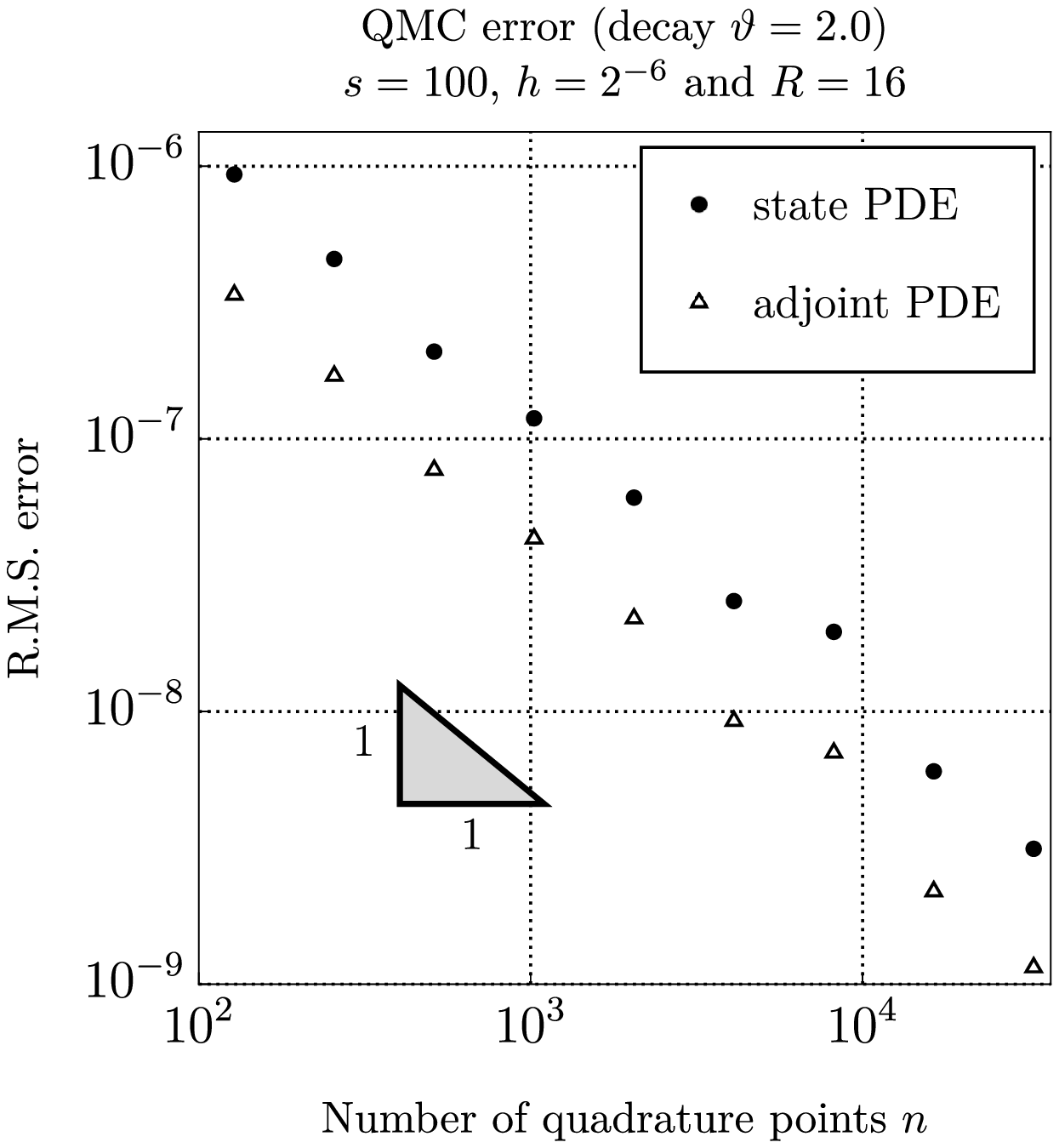}}}
\caption{The computed root-mean-square errors for the randomly shifted lattice rules.}\label{fig:qmc}
\end{figure}

\subsection{Optimal control problem}\label{subsectionNum4}

We consider the problem of finding the optimal control $z\in\mathcal{Z}$ that minimizes the functional \cref{eq:1.1}
subject to the PDE constraints~\cref{eq:1.2,eq:1.3}. 
~We choose $u_0(\bsx)=x_1^2-x_2^2$, set $\vartheta=1.5$, and fix the space of admissible controls $\mathcal{Z}= \{z \in L^2(\Omega)\,:\, z_{\min} \leq z \leq z_{\max}\, \text{ a.e.~in }\Omega\}$ with 
$$
z_{\min}(\pmb x) =  
\begin{cases}
0 & \pmb x \in \left[\tfrac{1}{8},\tfrac{3}{8}\right] \times \left[\tfrac{5}{8},\tfrac{7}{8}\right],\\
0 & \pmb x \in \left[\tfrac{5}{8},\tfrac{7}{8}\right] \times \left[\tfrac{5}{8},\tfrac{7}{8}\right],\\
-1 & \text{otherwise}
\end{cases}
\qquad \text{and} \qquad
z_{\max}(\pmb x) =  
\begin{cases}
0 & \pmb x \in \left[\tfrac{1}{8},\tfrac{3}{8}\right] \times \left[\tfrac{1}{8},\tfrac{3}{8}\right],\\
0 & \pmb x \in \left[\tfrac{5}{8},\tfrac{7}{8}\right] \times \left[\tfrac{1}{8},\tfrac{3}{8}\right],\\
1 & \text{otherwise}.
\end{cases}
$$ We use finite elements with mesh width $h=2^{-6}$ to discretize the spatial domain $\Omega=(0,1)^2$. The integrals over the parametric domain $\Xi$ are discretized using a lattice rule with a single fixed random shift with $n=2^{15}$ points and  the truncation dimension $s=2^{12}$.

We consider the regularization parameters $\alpha\in\{0.1,0.01\}$ for the minimization problem. To minimize the discretized target functional, we use the projected gradient descent algorithm (\cref{alg:projected gradient descent}) in conjunction with the projected Armijo rule (\cref{alg:projected Armijo}) with $\gamma = 10^{-4}$ and $\beta = 0.5$. For both experiments, we used $z_0(\bsx)=P_{\mathcal Z}(x_2)$ as the initial guess and track the averaged least square difference of the state $u$ and the target state $u_0$. The results 
are displayed in \cref{fig:frechet1}. 
We observe that for a larger value of $\alpha$ the algorithm converges faster and the averaged difference between the state $u$ and the target state $u_0$ increases.

\begin{figure}[!h]
\centering
\subfloat{{\includegraphics[height=.45\textwidth]{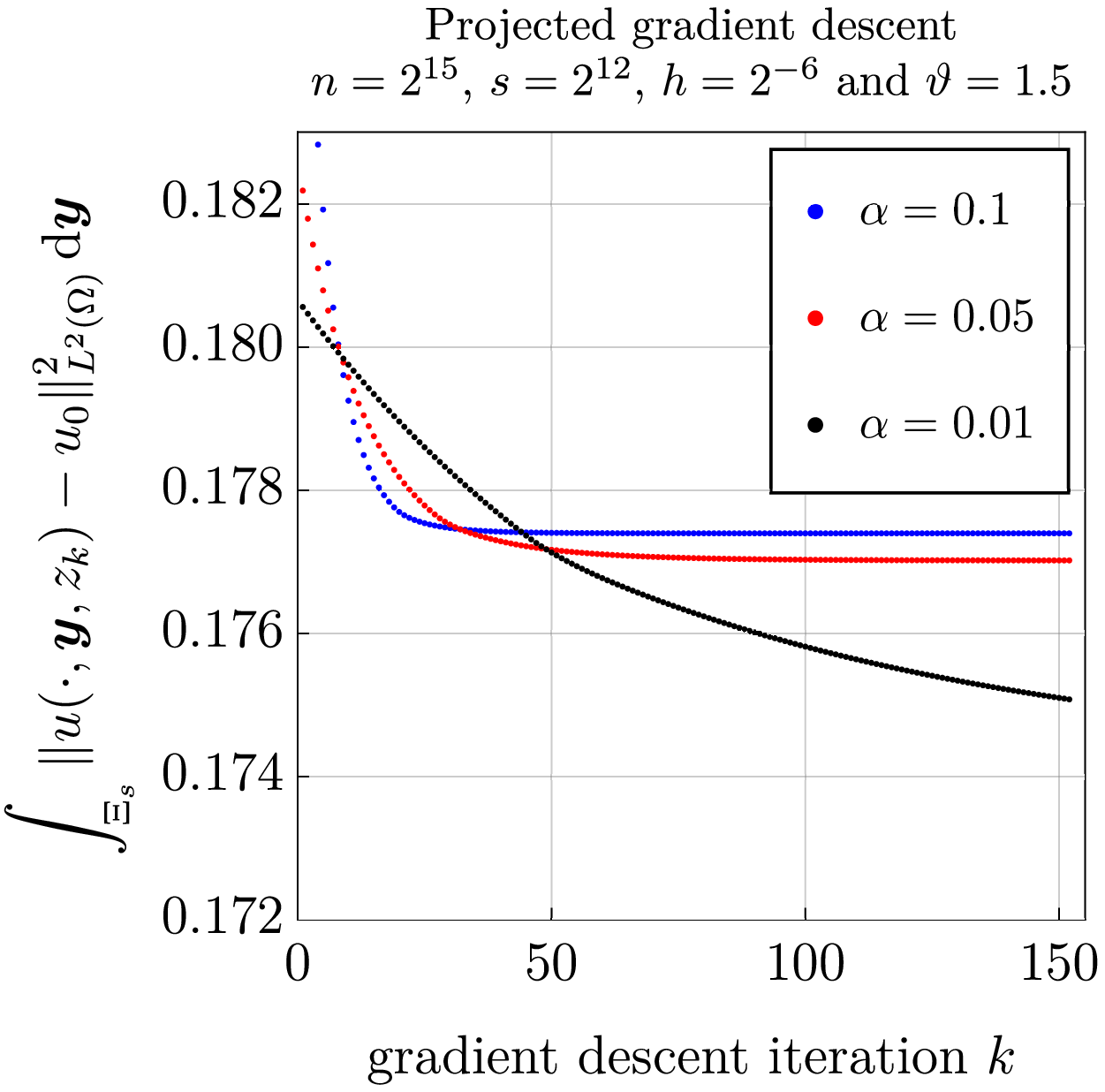}}}\qquad\subfloat{{\includegraphics[width=.45\textwidth]{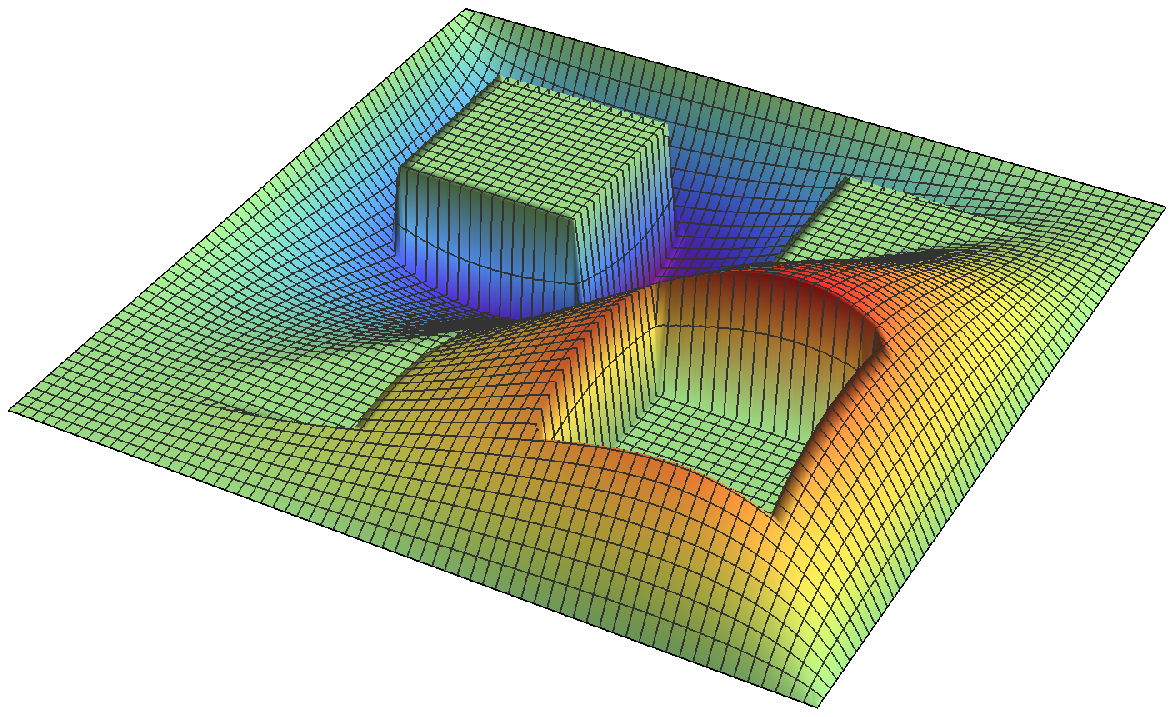}}}
\caption{Left: Averaged least square difference of the state $u$ and the target state $u_0$ at each step of the projected gradient descent algorithm for different values of the regularization parameter $\alpha$. Right: The control corresponding to $\alpha = 0.1$ after $152$ projected gradient descent iterations.}\label{fig:frechet1}
\end{figure}

The same behaviour is observed in the unconstrained case with $\mathcal{Z} = L^2(\Omega)$. We fix the same parameters as before and use the gradient descent algorithm \cref{alg:gradient descent} together with the Armijo rule \cref{alg:Armijo} with $\gamma = 10^{-4}$ and $\beta = 0.5$. We choose $z_0(\bsx)=x_2$ as the initial guess and track the averaged least square difference of the state $u$ and the target state $u_0$. The results 
are displayed in \cref{fig:frechet2}.
\begin{figure}[!h]
\centering
\subfloat{{\includegraphics[height=.45\textwidth]{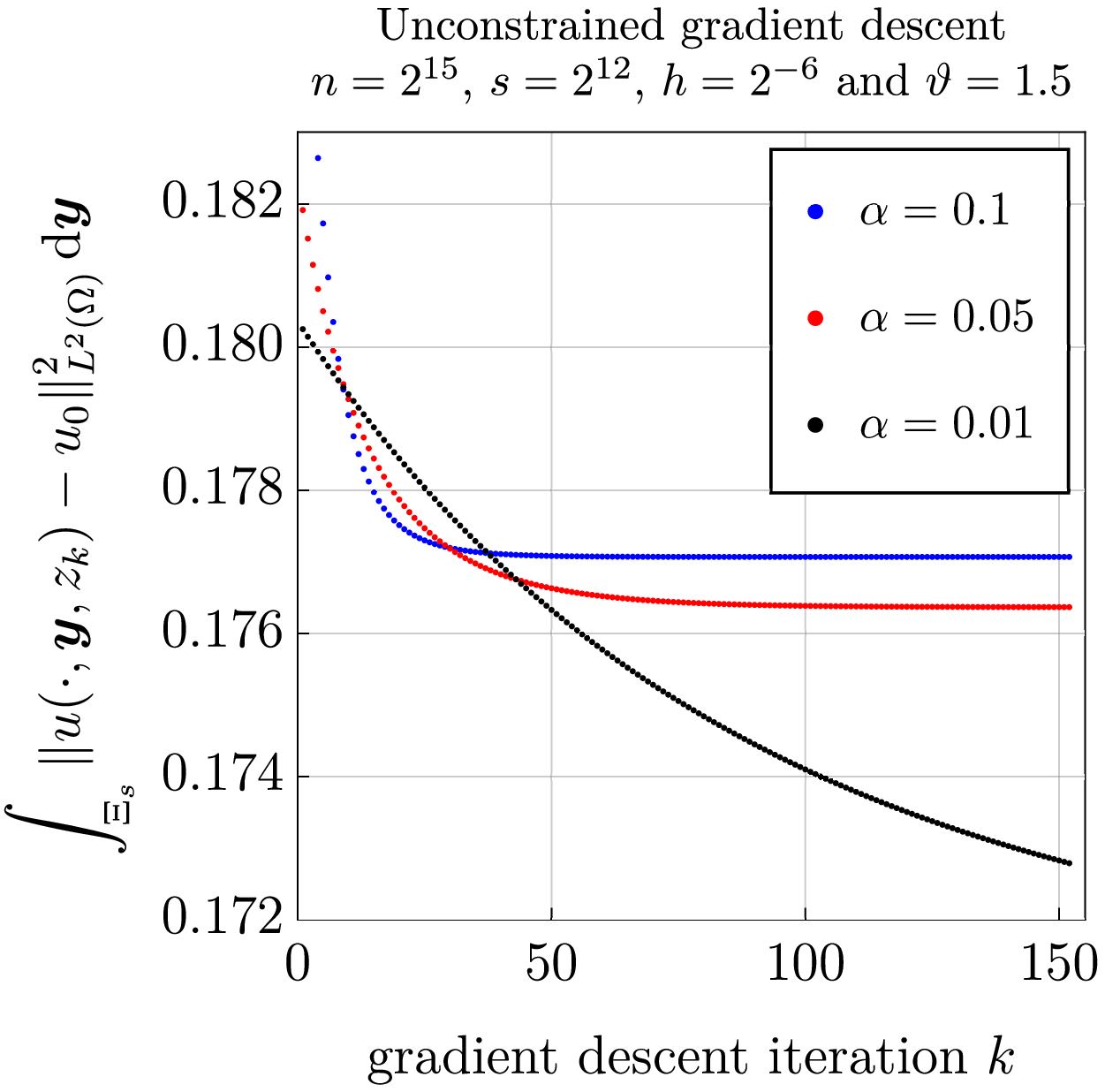}}}\qquad\subfloat{{\includegraphics[width=.45\textwidth]{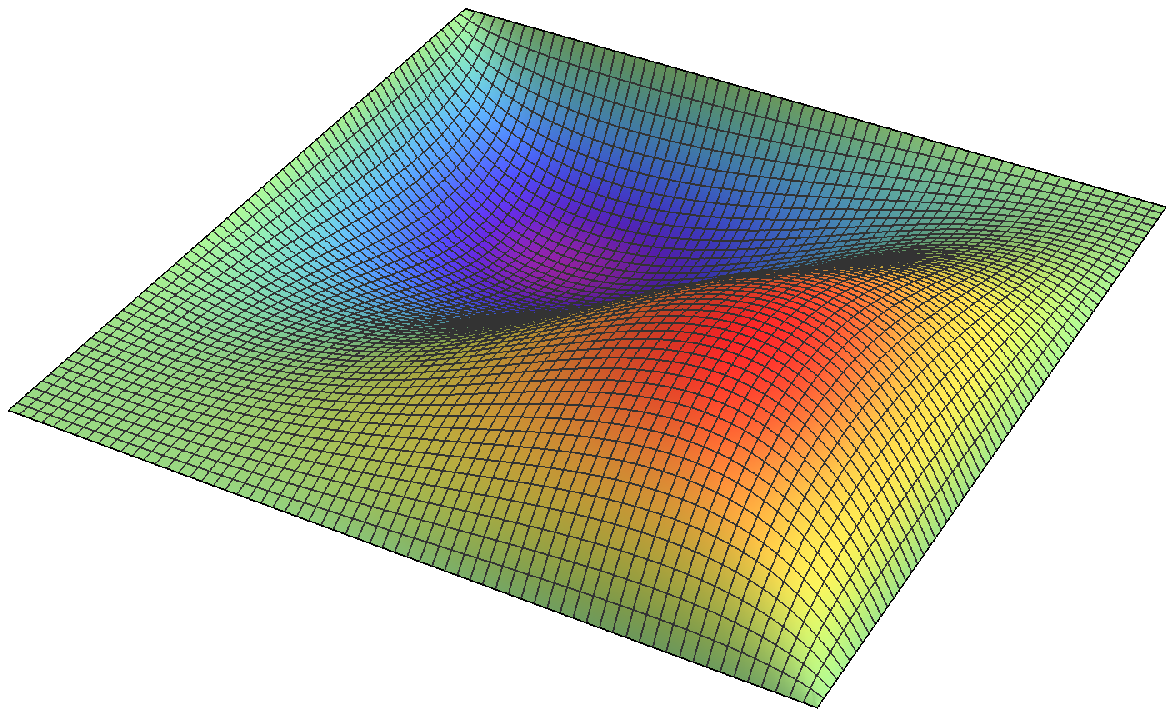}}}
\caption{Left: Averaged least square difference of the state $u$ and the target state $u_0$ at each step of the gradient descent algorithm for different values of the regularization parameter $\alpha$. Right: The control corresponding to $\alpha = 0.1$ after $152$ gradient descent iterations.}\label{fig:frechet2}
\end{figure}


\section{Conclusion and future work}

We presented a specially designed quasi-Monte Carlo method for the robust optimal control problem. Our proposed method provides error bounds for the approximation of the stochastic integral, which do not depend on the number of uncertain variables. Moreover, the method results in faster convergence rates compared to Monte Carlo methods. In addition our method preserves the convexity structure of the optimal control problem due to the nonnegative (equal) quadrature weights. Moreover we presented error estimates and convergence rates for the dimension truncation and the finite element discretization together with confirming numerical experiments.

Based on this work and motivated by \cite{vanBarelVandewalle}, multilevel \cite{AUH,KSSSU,KSS2015} and multi-index \cite{DFS} strategies can be developed in order to further decrease the computational burden. Furthermore the regularity results of this work can be used for the application of higher order QMC rules \cite{DKGNS}. Depending on the application it may also be of interest to consider different objective functions, e.g., the conditional value-at-risk, a combination of the expected value and the variance or different regularization terms. In addition it remains to extend the theory to a class of different forward problems such as affine parametric operator equations \cite{KunothSchwab,KunothSchwab2,Schwab} and different random fields as coefficients of the PDE system \cite{GKNSSS,KKS,KSSSU}. Other possible improvements include more sophisticated optimization algorithms such as Newton based methods.

\appendix

\bibliographystyle{siamplain}
\bibliography{references}

\end{document}